\documentclass[microtype]{gtpart}

\usepackage{amssymb,pstricks,amscd,epsfig}
\usepackage{physics}
\usepackage[utf8]{inputenc}
\usepackage{graphicx}

\usepackage{overpic}

\begin{document}
\newtheorem{cor}{Corollary}[section]
\newtheorem{theorem}[cor]{Theorem}
\newtheorem{prop}[cor]{Proposition}
\newtheorem{lemma}[cor]{Lemma}
\newtheorem{sublemma}[cor]{Sublemma}
\newtheorem{stat}[cor]{Statement}
\theoremstyle{definition}
\newtheorem{definition}[cor]{Definition}
\theoremstyle{remark}
\newtheorem{remark}[cor]{Remark}
\newtheorem{example}[cor]{Example}
\newtheorem{question}[cor]{Question}
\newtheorem{conjecture}[cor]{Conjecture}

\newcommand{\cA}{{\mathcal A}}
\newcommand{\cB}{{\mathcal B}}
\newcommand{\cC}{{\mathcal C}}
\newcommand{\cD}{{\mathcal D}}
\newcommand{\cE}{{\mathcal E}}
\newcommand{\cF}{{\mathcal F}}
\newcommand{\cG}{{\mathcal G}}
\newcommand{\cM}{{\mathcal M}}
\newcommand{\cL}{{\mathcal L}}
\newcommand{\cN}{{\mathcal N}}
\newcommand{\cP}{{\mathcal P}}
\newcommand{\cQ}{{\mathcal Q}}
\newcommand{\cS}{{\mathcal S}}
\newcommand{\cV}{{\mathcal V}}
\newcommand{\cT}{{\mathcal T}}
\newcommand{\cW}{{\mathcal W}}
\newcommand{\cCP}{{\mathcal C\mathcal P}}
\newcommand{\cML}{{\mathcal M\mathcal L}}
\newcommand{\cFML}{{\mathcal F\mathcal M\mathcal L}}
\newcommand{\cGH}{{\mathcal G\mathcal H}}
\newcommand{\cQF}{{\mathcal Q\mathcal F}}
\newcommand{\cMF}{{\mathcal M\mathcal F}}
\newcommand{\dwp}{d_{WP}}
\newcommand{\bC}{{\mathbb C}}
\newcommand{\bS}{{\mathbb S}}
\newcommand{\bP}{{\mathbb P}}
\newcommand{\bCP}{{\mathbb{CP}}}
\newcommand{\N}{{\mathbb N}}
\newcommand{\bR}{{\mathbb R}}
\newcommand{\bZ}{{\mathbb Z}}
\newcommand{\bH}{{\mathbb H}}
\newcommand{\Kt}{\tilde{K}}
\newcommand{\Mt}{\tilde{M}}
\newcommand{\dr}{{\partial}}
\newcommand{\betab}{\overline{\beta}}
\newcommand{\kappab}{\overline{\kappa}}
\newcommand{\pib}{\overline{\pi}}
\newcommand{\taub}{\overline{\tau}}
\newcommand{\gb}{\overline{g}}
\newcommand{\hb}{\overline{h}}
\newcommand{\ub}{\overline{u}}
\newcommand{\Bb}{\overline{B}}
\newcommand{\Kb}{\overline{K}}
\newcommand{\Sigmab}{\overline{\Sigma}}
\newcommand{\gd}{\dot{g}}
\newcommand{\hd}{\dot{h}}
\newcommand{\Id}{\dot{I}}
\newcommand{\Jd}{\dot{J}}
\newcommand{\diff}{\mbox{Diff}}
\newcommand{\isom}{\mathrm{Isom}}
\newcommand{\devb}{\overline{\mbox{dev}}}
\newcommand{\devt}{\tilde{\mbox{dev}}}
\newcommand{\vol}{\mbox{Vol}}
\newcommand{\hess}{\mathrm{Hess}}
\newcommand{\gr}{\mathrm{gr}}
\newcommand{\Gr}{\mathrm{Gr}}
\newcommand{\supp}{\mathrm{supp}}
\newcommand{\cb}{\overline{c}}
\newcommand{\db}{\overline{\partial}}
\newcommand{\hgr}{h_{gr}}
\newcommand{\Sigmat}{\tilde{\Sigma}}
\newcommand{\wt}[1]{\widetilde{#1}}

\newcommand{\al}{\alpha}
\newcommand{\eps}{\epsilon}

\newcommand{\cunc}{{\mathcal C}^\infty_c}
\newcommand{\cun}{{\mathcal C}^\infty}
\newcommand{\sig}{\sigma}
\newcommand{\dev}{\mathrm{dev}}
\newcommand{\hol}{\mathrm{hol}}
\newcommand{\PSL}{\mathrm{PSL}}
\newcommand{\dmin}{d_{\mathrm{min}}}
\newcommand{\dmax}{d_{\mathrm{max}}}
\newcommand{\Dom}{\mathrm{Dom}}
\newcommand{\dn}{d_\nabla}
\newcommand{\ded}{\delta_D}
\newcommand{\delmin}{\delta_{\mathrm{min}}}
\newcommand{\delmax}{\delta_{\mathrm{max}}}
\newcommand{\hmin}{H_{\mathrm{min}}}
\newcommand{\maxi}{\mathrm{max}}
\newcommand{\hull}{\mathrm{hull}}
\newcommand{\oL}{\overline{L}}
\newcommand{\oP}{{\overline{P}}}
\newcommand{\xb}{{\overline{x}}}
\newcommand{\yb}{{\overline{y}}}
\newcommand{\bRan}{\mathrm{Ran}}
\newcommand{\tgamma}{\tilde{\tau}}
\newcommand{\cotan}{\mbox{cotan}}
\newcommand{\area}{\mbox{Area}}
\newcommand{\lambdat}{\tilde\lambda}
\newcommand{\xt}{\tilde x}
\newcommand{\bCt}{\tilde C}
\newcommand{\St}{\tilde S}

\newcommand{\re}{\mathrm{Re}}
\newcommand{\sch}{\mathrm{Sch}}
\newcommand{\ric}{\mathrm{Ric}}
\newcommand{\scal}{\mathrm{Scal}}
\newcommand{\ext}{\mathrm{ext}}
\newcommand{\diam}{\mathrm{diam}\,}

\newcommand{\II}{I\hspace{-0.1cm}I}
\newcommand{\III}{I\hspace{-0.1cm}I\hspace{-0.1cm}I}
\newcommand{\note}[1]{{\small {\color[rgb]{1,0,0} #1}}}

\title[Delaunay circle patterns]{Properness for circle packings and Delaunay circle patterns on complex projective structures}

\author{Jean-Marc Schlenker}
\thanks{J.-M. S. was partially supported by University of Luxembourg IRP NeoGeo and by FNR projects INTER/ANR/15/11211745 and OPEN/16/11405402. J.-M. S. also acknowledges support from U.S. National Science Foundation grants DMS-1107452, 1107263, 1107367 ``RNMS: GEometric structures And Representation varieties'' (the GEAR Network).}
\address{University of Luxembourg,
Department of mathematics, 
University of Luxembourg, 
Maison du nombre, 6 avenue de la Fonte,
L-4364 Esch-sur-Alzette, Luxembourg
}
\email{jean-marc.schlenker@uni.lu}

\author{Andrew Yarmola}
\email{andrew.yarmola@uni.lu}

\date{v1, \today}

\begin{abstract}
  We consider circle packings and, more generally, Delaunay circle patterns --- arrangements of circles arising from a Delaunay decomposition of a finite set of points --- on surfaces equipped with a complex projective structure. Motivated by a conjecture of Kojima, Mizushima and Tan, we prove that the forgetful map sending a complex projective structure admitting a circle packing with given nerve (resp. a Delaunay circle pattern with given nerve and intersection angles) to the underlying complex structure is proper.
\end{abstract}

\maketitle

\section{Introduction and main results}

\subsection{Outline of the main results}

Circle packings have been of mathematical interest throughout history, going as far back as the works of Apollonius of Perga. A circle packing is, informally, a set of disjoint disks and the nerve of a circle packing is a graph which records the tangency relations between these disks. One of the main historic results of interest is the Koebe Circle Packing Theorem \cite{koebe}: any 3-connected graph on the sphere can be realized as the nerve of a circle packing, which is unique up to M\"obius transformations. This beautiful theorem was rediscovered and generalized by Thurston, who also related it to hyperbolic geometry and Andreev's Theorem \cite{Andreev,Andreev-ideal}. Thurston later showed that circle packings provide a discrete version of the Riemann Mapping Theorem, leading to new developments and applications, see e.g. \cite{rodin-sullivan,bowers-stephenson_conformalI,bobenko-springborn,KSS} or \cite{stephenson-notices} for a nice survey.

In addition to the sphere, circle packings have been extensively studied for Euclidean and hyperbolic surfaces \cite[Sections 13.6 and 13.7]{thurston-notes}. However, the notion of circle packing can be considered for``weaker'' geometric structures, namely, complex projective structures (see Section \ref{ssc:cp}). Throughout this paper, $S$ will be an oriented closed surface of genus $g\geq 2$, generally equipped with a complex projective structure $\sigma$. We will let $\cC$ denote the space of complex projective structures on $S$, considered up to isotopy. 

A complex projective structure has an underlying complex structure, leading to a conjecture of Kojima, Mizushima and Tan \cite{KMT,KMT3}, rephrased below.

\begin{conjecture} \label{cj:main}
Let $\tau$ be a polygonal cell decomposition of $S$, and let $c$ be a complex structure on $S$. There is a unique complex projective structure $\sigma$ on $S$ with underlying complex structure $c$ admitting a circle packing with nerve $\tau$.
\end{conjecture}

When $S$ is the sphere, this conjecture reduces to the Koebe Circle Packing Theorem.

A natural strategy to prove this conjecture is to consider the space $\cC_\tau$ of complex projective structures on $S$ equipped with a circle packing with nerve $\tau$ (see Section \ref{ssc:packings} for a proper definition), and the map $f_\tau:\cC_\tau\to \cT$ sending a complex projective structure to the underlying complex structure, where $\cT$ denotes the Teichm\"uller space of $S$. Conjecture \ref{cj:main} would follow from the following:
\begin{enumerate}
\item $\cC_\tau$ is a manifold of dimension $6g-6$,
\item $f_\tau$ is locally injective,
\item $f_\tau$ is proper,
\item $\cC_\tau$ is connected.
\end{enumerate}
Points (1)-(3) and invariance of domain would imply that $f_\tau$ is a proper local homeomorphism, and therefore a covering map. Here, we focus on the third point and our first contribution is the proof of this properness result, see Theorem \ref{tm:packing}. We believe that points (1) and (2) might be amenable to other, more analytic methods in the spirit of \cite{scannell-wolf}. Note however, that properness alone and a degree-type argument might be sufficient to prove existence in Conjecture \ref{cj:main}.

Our second contribution is to put Conjecture \ref{cj:main} in the somewhat more general context of {\em Delaunay circle patterns}. A Delaunay circle pattern is, informally, the pattern of circles defined by the Delaunay decomposition of a finite set of points, see Section \ref{ssc:delaunay}. Circle packings can be seen as special cases of Delaunay circle patterns with all intersection angles equal to $\pi/2$. We believe that Conjecture \ref{cj:main} can be extended from circle packings to Delaunay circle patterns, and the properness result we prove is actually for Delaunay circle patterns, see Theorem \ref{tm:delaunay}.

The proofs of our properness results are based on 3-dimensional hyperbolic geometry. To each complex projective structure $\sigma$ on $S$, one can associate a {\em hyperbolic end} $E(\sigma)$ having $(S,\sigma)$ as its ideal boundary, see Section \ref{ssc:ends}. We show that a Delaunay circle pattern on $(S,\sigma)$ corresponds to an ``ideal polyhedron'' in $E(\sigma)$. Our third contribution is to provide a description of circle packings (resp. Delaunay circle patterns) in terms of complete hyperbolic metrics of finite area on the punctured surface $S_{g,n}$ satisfying some ``balancing'' properties, see Proposition \ref{prop:balanced1} (resp. Proposition \ref{prop:balanced2}). Although this description is not formally necessary for our properness results, we believe that it helps with the exposition and is of independent interest. In particular, we hope that ``balancing'' properties can lead to proofs of points (1) and (2) above.

\subsection{Complex projective structures and round disks}
\label{ssc:cp}

A {\em complex projective structure} $\sig$ on $S$ is a maximal atlas of charts to $\bCP^1$ where the transition maps are restrictions of M\"obius transformations. One can think of $\sig$ as a pair $(\dev_\sig, \hol_\sig)$, where $\dev_\sig : \wt{S} \to \bCP^1$ is the developing map and $\hol_\sig : \pi_1 S \to \PSL(2,\bC)$ is the holonomy. We let $\cC$ denote the space of complex projective structures on $S$, considered up to isotopy. See, \cite{dumas-survey} for more details.

Since M\"obius transformations send circles to circles in $\bCP^1$, there is a well-defined notion of disks and circles on $(S,\sig)$. In particular,

\begin{definition}  A \emph{round disk} is a connected closed subset $D$ of $\wt{S}$ such that $\left.\dev_{\sig}\right|_{D}$ is injective and $\dev_\sig(D)$ is a closed disk in $\bCP^1$. A {\it round circle} is the boundary of a round disk. When the covering map $\wt{S} \to S$ is injective on a round disk, we call the {\it image} in $S$ an {\it embedded disk} (resp. {\it embedded circle}).
\end{definition}

\subsection{Circle packings}
\label{ssc:packings}

A {\it circle packing} on $(S,\sigma)$ is a finite collection $\{D_1,\cdots, D_n\}$ of \emph{embedded disks} in $(S,\sigma)$ with disjoint interiors such that each connected component of $S\setminus \bigcup_{i=1}^n D_i$ is M\"obius equivalent to an ideal hyperbolic polygon.

\begin{figure}[htb]
\begin{center}
\begin{overpic}[scale=.4]{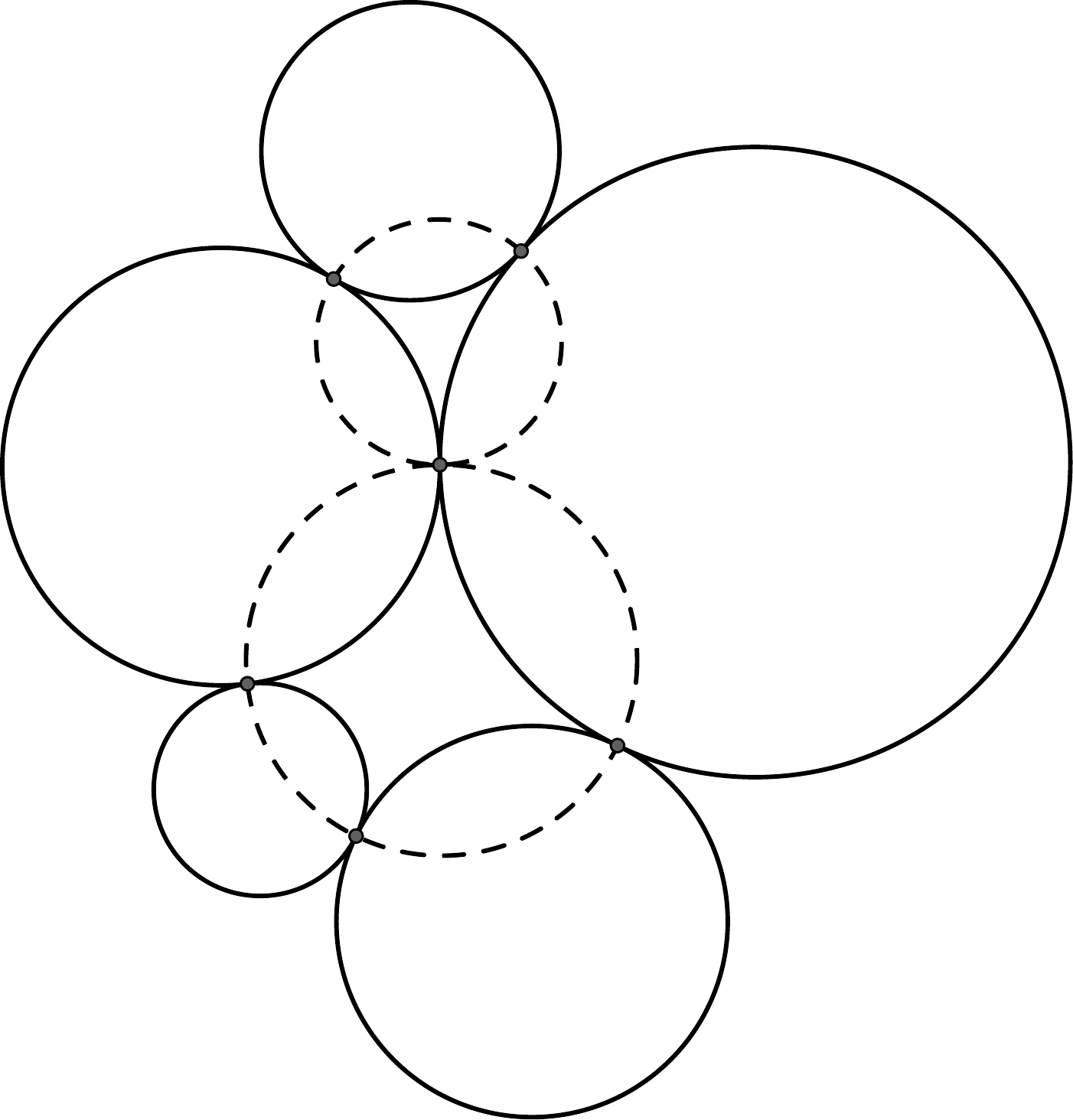}
\end{overpic}
\end{center}
\caption{A piece of a circle packing with dual circles represented with dashes.}
\label{fig:circle_pack}
\end{figure}

It follows from this definition that each complementary region is simply connected and its vertices lie on an embedded circle in $(S,\sigma)$. We call these \emph{dual circles}, see Figure \ref{fig:circle_pack}.

\begin{definition}
The {\it nerve} $\tau_\cP$ of a circle packing $\cP$ is a cell decomposition of $S$ with a vertex for each element of $\cP$, an edge between two vertices for each tangency point between corresponding disks, and a face for every complementary region. See Figure \ref{fig:circle_nerve}.
\end{definition}

\begin{figure}[htb]
\begin{center}
\begin{overpic}[scale=.4]{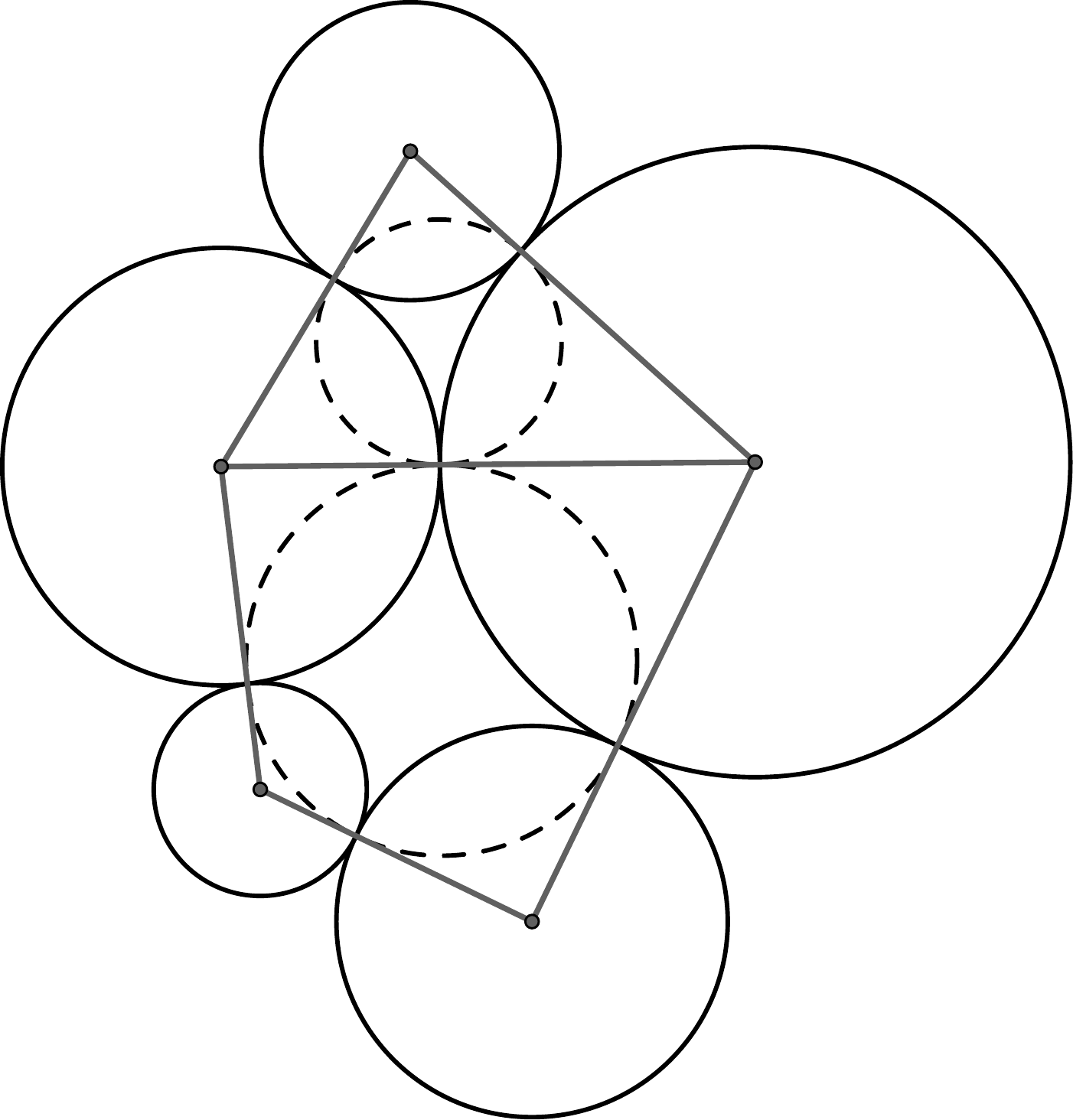}
\end{overpic}
\end{center}
\caption{A piece of a nerve of a circle packing.}
\label{fig:circle_nerve}
\end{figure}

In fact, $\tau_\cP$ is a special type of cell decomposition where the $2$-cells are polygons.

\begin{definition} A \emph{polygonal cell decomposition} $\eta$ of $S$ is a partition of $S$ into cells such that
\begin{enumerate}
\item each $2$-cell is a polygon.
\item every $0$-cell meets at least three $1$-cells.
\item the 1-skeleton, $\eta_1$, lifts to a simple graph in $\wt{S}$.
\end{enumerate}
Condition (3) implies $\wt{\eta}_1$ will have no loops or double edges.
\end{definition}

\subsection{A properness result for circle packings}

In the previous section, we assigned a cell decomposition to a circle packing. To go in reverse, we consider the space of all complex projective structures that admit a circle packing with fixed combinatorics.

\begin{definition}
Let $\tau$ be a polygonal cell decomposition of $S$. Define $\cC_\tau$ to be the space of pairs $(\sigma, \cP)$ where $\sigma\in \cC$ and $\cP$ is a circle packing on $(S,\sigma)$ with nerve isotopic to $\tau$. The topology is inherited from the bundle of round disks on $\cC$.
\end{definition}
  
Let $\cT$ denote the Teichm\"uller space of $S$. There is a forgetful map $f:\cC\to \cT$ sending a complex projective structure to the underlying complex structure. This gives the map $f_\tau=f\circ \Pi_1:\cC_\tau\to \cT$, where $\Pi_1$ denotes the projection on the first factor. Kojima, Mizushima and Tan \cite{KMT,KMT3} made the following conjecture, which is a more precise version of Conjecture \ref{cj:main}.

\begin{conjecture} \label{cj:kmt}
Let $\tau$ be a triangulation of $S$ that lifts to a simple graph in $\wt{S}$. Then $f_\tau:\cC_\tau\to \cT$ is a homeomorphism.
\end{conjecture}

The simplest example of this statement is obtained when $S$ is a sphere, reducing the conjecture to the classical Koebe Circle Packing Theorem \cite{koebe}. The conjecture was also proven in \cite{mizushima} for the case of the hexagonal one vertex triangulation of the torus.

A possible strategy towards Conjecture \ref{cj:kmt} is to use a ``deformation'' approach based on the following four steps:
\begin{enumerate}
\item $\cC_\tau$ is a manifold of dimension $6g-6$,
\item $f_\tau$ is locally injective.
\item $f_\tau$ is proper.
\item $\cC_\tau$ is connected.
\end{enumerate}
A proof of steps (1)-(3) and invariance of domain would imply that $f_\tau$ is a covering map, and, adding step (4), would give that $f_\tau$ is a homeomorphism. Steps (1) and (4) could be observed by showing that $\cC_\tau$ is a smooth variety of the correct dimension, see \cite[Main Theorem 2]{KMT} for the one vertex case. Here, we prove step (3).

\begin{theorem} \label{tm:packing}
Let $\tau$ be a polygonal cell decomposition of $S$. Then $f_\tau$ is proper.
\end{theorem}

This result was previously shown for surfaces of genus at least $2$ when $\tau$ has only one vertex, see \cite[Theorem 1.1]{KMT2}. For us, Theorem \ref{tm:packing} is a special case of Theorem \ref{tm:delaunay} below, as explained there. A closely related result was recently proved by E. Danenberg for combinatorially interesting $\tau$, see \cite{dannenberg_thesis}.

\subsection{Delaunay cell decompositions and circle patterns}
\label{ssc:delaunay}

In this section, we extend the notion of circle packing to a special type of circle pattern called Delaunay. Basically, a Delaunay circle pattern is the pattern of circles obtained from a Delaunay decomposition of a set of points in $(S,\sigma)$. Recall that a Delaunay decomposition of a set of points $V$ on $(S, \sigma)$ is a polygonal cell decomposition $\eta$ with $\cV(\eta) = V$ such the vertices of each polygon lie on the boundary of a round disk containing no other elements of $\cV(\eta)$ in its interior. 

Thurston noted that given a circle packing, one can obtain a nice Delaunay circle pattern with all intersection angles equal to $\pi/2$, as described in the next example.

\begin{example}\label{ex:circle_delaunay}
Let $\cP=\{ D_1,\ldots, D_n\}$ be a circle packing on $(S,\sigma)$ with nerve $\tau$. By definition, each connected component $P$ of $S\setminus \cup_{i=1}^n D_i$ is M\"obius equivalent to an ideal hyperbolic polygon and therefore defines a dual disk $D_P$ containing $P$, see Figure \ref{fig:circle_pack}. Since the boundary arcs of $P$ are perpendicular to $\partial D_P$, each dual disk is perpendicular to the circle packing disks which it intersects. Let $\cD_\cP$ be the union of $\cP$ and the set of dual disks of $\cP$. Then $\cD_\cP$ is a Delaunay circle pattern on $(S,\sigma)$ with all intersection angles are equal to $\pi/2$.
\end{example}

To make this definition precise in general, we will work in the universal cover $\wt S$ of $S$. For a finite collection $\cD$  of round disks in $\sigma$, let $\wt{D}$ be the set of all lifts to $\wt{S}$ and define

$$ \wt{V}_\cD = \{ x \in S \mid x \in  \partial D\cap \partial D' \text{ for distinct } D,D' \in \wt{\cD} \text{ and } x\not\in \mathrm{int}\left(D''\right) \; \forall D''\in \wt{\cD} \}.$$

By definition, $\wt{V}_\cD$ is $\pi_1 S$-invariant and we have a nice projection $V_\cD = \wt{V}_\cD/\pi_1 S \subset S$. 

\begin{definition} A {\it Delaunay circle pattern} on $(S,\sigma)$ is a finite collection $\cD$ of round disks such that 
\begin{enumerate}
\item $V_\cD$ is finite.
\item $S =  \bigcup_{D\in \cD}D$.
\item For every $D \in \wt{\cD}$, let $P_{\partial D} = \wt{V}_\cD \cap \partial D = \{x_1, \ldots x_k\}$ be cyclically ordered. We want $3 \leq k < \infty$ and for every $1 \leq i \leq k$ there exists $D_i \in  \wt{\cD}$ such that $\{x_i, x_{i+1}\} =  \partial D \cap \partial D_i$, where $x_{k+1} = x_1$.  We call $\{D,D_i\}$ a {\em cutting disk pair}.
\end{enumerate}
See Figure \ref{fig:delaunay_pattern} for an example.
\end{definition}

\begin{figure}[htb]
\begin{center}
\begin{overpic}[scale=.6]{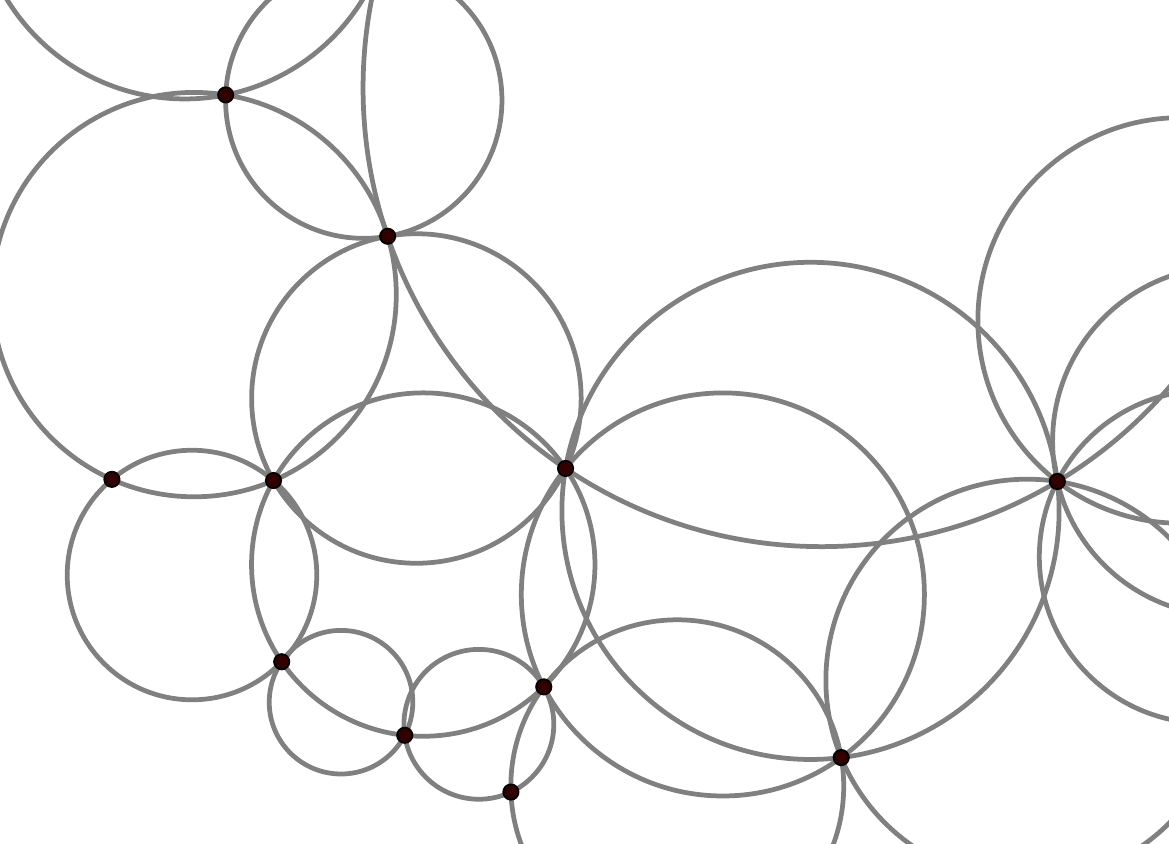}
\end{overpic}
\end{center}
\caption{A piece of a Delaunay circle pattern $\cD$ along with the associated vertex set $\wt{V}_\cD$.}
\label{fig:delaunay_pattern}
\end{figure}

Given a Delaunay circle pattern $\cD$, we build a polygonal cell decomposition $\wt{\eta}_\cD$ of $\wt{S}$ as follows. The $0$-cells of $\wt{\eta}_\cD$ are just the elements of $\wt{V}_\cD$. There is a 1-cell between $x,x'\in \wt{V}_\cD$ if and only if $x$ and $x'$ occur in cyclic order on some $\partial D$ for $D \in \wt{\cD}$. The $2$-cells are the polygons inside each $D \in \wt{\cD}$ cut out by the edges. The {\it associated polygonal cell decomposition} $\eta_\cD$ for $\cD$ is obtained as $\wt{\eta}_\cD/\pi_1S$. It is important to note that the 1-skeleton of $\eta_\cD$ is filling on $S$ by construction.
  
\begin{figure}[htb]
\begin{center}
\begin{overpic}[scale=.6]{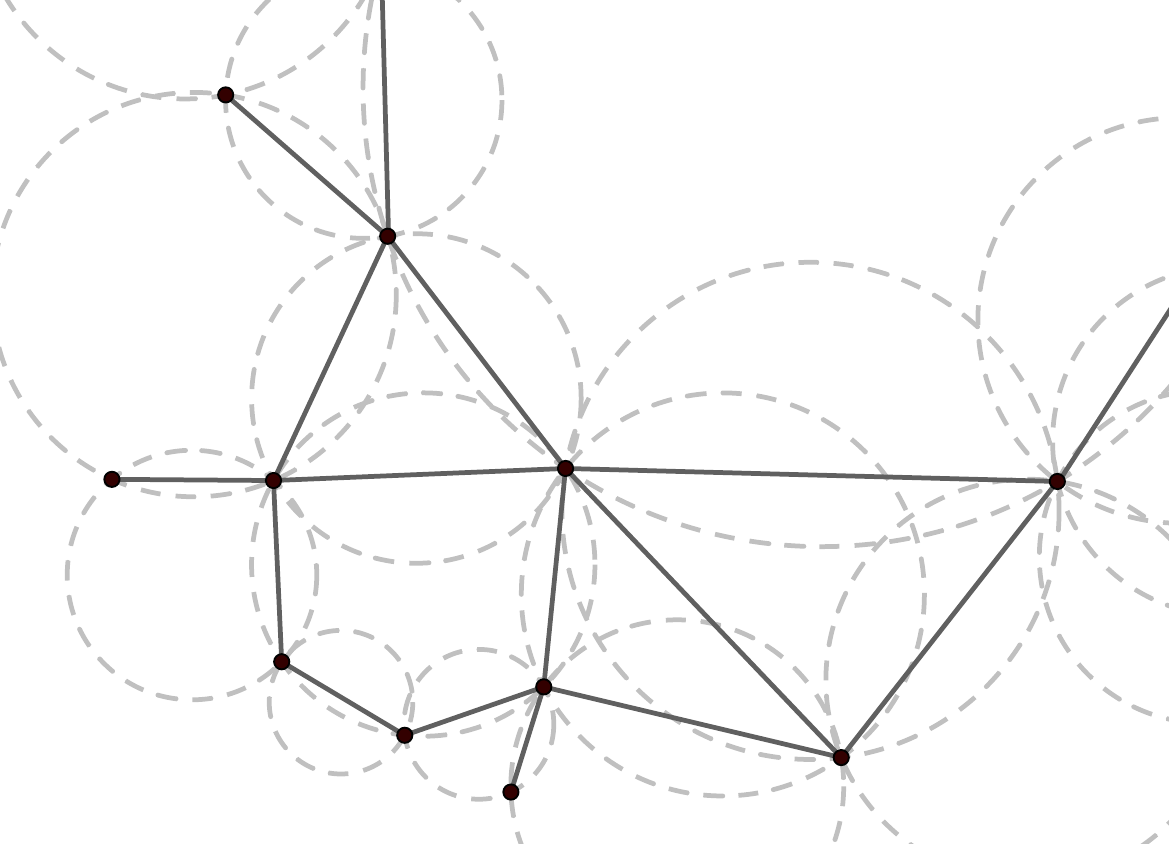}
\end{overpic}
\end{center}
\caption{A piece of a Delaunay circle pattern $\cD$ along with the associated decomposition $\wt{\eta}_\cD$.}
\label{fig:delaunay_tangency}
\end{figure}

These circle patterns are called Delaunay because of the following condition : any three consecutive vertices of a polygonal face lie on the boundary of an (immersed) round disk in $(S,\sig)$ and the interior of this disk contains no vertex of $\eta_\cD$. Additionally, each vertex and face of $\eta_\cD$ meets at least $3$ edges.

\begin{definition} The {\it nerve} of a Delaunay circle pattern is $\eta^*_\cD$, the dual polygonal cell decomposition of $\eta_\cD$.
\end{definition}

Note that every vertex of $\eta^*_\cD$ correspond to a disk in $\cD$. Additionally, if $\cP$ is a circle packing and $\cD_\cP$ is the associated Delaunay circle pattern as in Example \ref{ex:circle_delaunay}, then the $1$-skeleton of $\eta^*_{\cD_\cP}$ is bipartite with each edge connecting a disk of $\cP$ to some overlapping dual disk. See Remark \ref{rk:relation} for more details.

For Delaunay circle patters, we want to keep track of the angle information between overlapping circles. We define the {\em angle function} $\theta_\cD: \cE(\eta^*_\cD) \to (0,\pi)$, where $\cE(\eta^*_\cD)$ is the edge set of $\eta^*_\cD$, as follows. Given $e \in \cE(\eta^*_\cD)$, the endpoints of a lift $\wt{e}$ to $\wt{S}$ correspond to two round disks $D_e^0, D_e^1 \in \wt{\cD}$. We let $\theta_\cD(e)$ be the angle between the outward normals of $D_e^0$ and $D_e^1$ at a point of $\partial D_e^0 \cap \partial D_e^1$. Notice that this is well defined up to M\"obius transformations, as they are conformal. In general, the angle function satisfies some nice conditions.

\begin{lemma}\label{lem:angles} The angle function $\theta_\cD: \cE(\eta^*_\cD) \to (0, \pi)$ satisfies the following properties.
  \begin{enumerate}
    \item For each face $f$ of $\eta^*_\cD$, $\sum_{e \in \partial f} \theta_\cD(e)=2\pi$.
    \item For each homotopically trivial non-backtracking closed edge path $[e_1,\cdots, e_n]$ in $\eta^*_\cD$ which does not bound a face, $\sum_{i=1}^n \theta_\cD(e_i)>2\pi$.
  \end{enumerate}
\end{lemma}

The proof can be found in Section \ref{ssc:delaunay}.

We say that a pair $(\eta, \theta)$ of a polygonal cell decomposition $\eta$ of $S$ and a function $\theta : \cE(\eta^*) \to (0, \pi)$ is { \it admissible} if $\theta$ satisfies the conclusions of Lemma \ref{lem:angles}.

Bobenko and Springborn have show that an admissible pair can always be realized on a unique Fuchsian structure. Recall that a complex projective structure is {\em Fuchsian} if the image of the developing map is a disk in $\bCP^1$, or, equivalently, if it is the underlying complex projective structure of a hyperbolic metric.

\begin{theorem}[{{\cite[Theorem 4]{bobenko-springborn}}}]\label{tm:bs}
  \label{thm:satisfy_angles}
If $(\eta, \theta)$ is admissible, then there exists a unique marked {\it hyperbolic} structure $(S, \sig_h)$ and a Delaunay circle pattern $\cD$ on $(S,\sig_h)$ such that $\eta = \eta_\cD$ and $\theta = \theta_\cD$.
\end{theorem}

\subsection{A properness result for Delaunay circle patterns}

Let $(\eta, \theta)$ be an admissible pair. We denote by $\cC_{\eta,\theta}$ the space of pairs $(\sigma, \cD)$ where $\sigma\in \cC$ is a complex projective structure, and $\cD$ is a Delaunay circle pattern on $(S,\sigma)$ with $\eta_\cD = \eta$ and $\theta_\cD = \theta$. The topology on the second factor is inherited from the bundle of round disks on $(S, \sig)$. Consider the function $f_{\eta,\theta}=f\circ \Pi_1:\cC_{\eta,\theta} \to \cT$.

We are motivated by the following question, extending Conjecture \ref{cj:kmt}.

\begin{question} \label{q:eta}
 Let $(\eta, \theta )$ be an admissible pair. Is $f_{\eta,\theta}$ a homeomorphism? 
\end{question}

Note that Theorem \ref{tm:bs} above shows that $\Pi_1(\cC_{\eta,\theta})$ intersects the Fuchsian locus of $\cC$ in a unique point (see also \cite[Theorem 4.23]{ideal}). In Section \ref{ssc:delaunay}, we show that only admissible angle functions can be associated to Delaunay circle patterns.

We can now state our main result.

\begin{theorem} \label{tm:delaunay}
Let $(\eta, \theta )$ be an admissible pair. Then $f_{\eta,\theta}:\cC_{\eta,\theta}\to \cT$ is proper.
\end{theorem}

We will show in Section \ref{ssc:from} that a positive answer to Question \ref{q:eta} would prove Conjecture \ref{cj:kmt}, and that Theorem \ref{tm:delaunay} implies Theorem \ref{tm:packing}.

\section{From Delaunay circle patterns to circle packings}
\label{ssc:from}
Let us focus on the relationship between circle packings and Delaunay circle patterns.
\begin{remark} \label{rk:relation}
  A Delaunay circle packing $\cD$ is obtained from a circle packing $\cP$ by adding the dual family of disks if and only if:
 \begin{enumerate}
 \item  the intersection angles of $\cD$ are all equal to $\pi/2$, and
 \item  the 1-skeleton of $\eta^*_\cD$ is bipartite.
 \end{enumerate}
\end{remark}

\begin{proof}
Let $\cP = \{D_1, \ldots, D_n\}$ be a circle packing on a complex projective structure $\sig$ on $S$. As in Example \ref{ex:circle_delaunay}, we obtain $\cD = \cD_\cP$ by adding the dual disks to $\cP$. Because each complement of $S\setminus \cup_{i=1}^n D_i$ is M\"obius equivalent to an ideal hyperbolic polygon, the angles between overlapping disks in $\cD$ are all $\pi/2$. Further, an overlap only occurs between a disk of $\cP$ and some dual disk. As vertices of $\eta^*_\cD$ correspond to disks, we can naturally partition them into the ``originals'' from $\cP$ and the duals, making the $1$-skeleton of $\eta^*_\cD$ bipartite.

Conversely, let $\eta^*$ be a polygonal cell decomposition of $S$ with a bipartite decomposition $\cV(\eta^*) = \cV^0 \cup \cV^1$ of the vertices. Let $\cD$ be a Delaunay circle pattern with $\eta^* = \eta^*_\cD$ and $\theta_\cD(e) = \pi/2$ for all edges. It follows from point (1) of Lemma \ref{lem:angles} that all faces of $\eta^*$ are 4-gons.
Thus, for each face, two opposite vertices correspond to elements of $\cV^0$ and the other two to elements of $\cV^1$. The corresponding four disks of $\wt{\cD}$ meet at a point, with each pair of opposite disks being tangent by the angle conditions. When applied to all faces of $\eta^*$, this implies that no disk of $D$ overlaps itself, so all disks are embedded. Let $\cP = \{ D \in \cD \mid D \text{ corresponds to a vertex in } \cV^0\}$. Then $\cP$ is a circle packing with tangency relation corresponding to the faces of $\eta^*$ and $\cD = \cD_\cP$.
\end{proof}

\begin{remark} \label{rk:nerves} Notice that the nerve of a circle packing is different from the nerve of a Delaunay circle pattern. However, one can get from one to the other combinatorially as follows. Let $\tau = \tau_\cP$ be the nerve of some circle packing $\cP$. Build a new polygonal cell decomposition $\eta$ by taking the vertices $\cV(\eta)$ to be the set of \emph{midpoints} of edges of $\tau$, then add an edge to $\eta$ for every pair of midpoints appearing cyclically around a vertex of $\tau$, and fill in the complimentary regions with faces. Notice that the faces of $\eta$ will correspond to both the vertices of $\tau$ and the faces of $\tau$. This ``midpoint'' cell decomposition $\eta$ will, in fact, be $\eta_{\cD_\cP}$. Taking the dual, we get $\eta^* = \eta^*_{\cD_\cP}$.
\end{remark}

We can now prove Theorem \ref{tm:packing} from Theorem \ref{tm:delaunay}, expanding on Example \ref{ex:circle_delaunay}.

\begin{proof}[Proof of Theorem \ref{tm:packing} assuming Theorem \ref{tm:delaunay}.]
 Let $\tau$ be a polygonal cell decomposition of $S$ and $\eta$ the ``midpoint'' cell decomposition constructed in Remark \ref{rk:nerves}. By Remark \ref{rk:relation}, $\cC_\tau$ can be identified with $\cC_{\eta,\pi/2}$, where, by abuse of notation, $\pi/2$ denotes the constant function on $\cE(\eta^*)$ with value $\pi/2$. In this identification, $f_\tau:\cC_\tau\to \cT$ is the function $f_{\eta,\pi/2}:\cC_{\eta,\pi/2}\to \cT$. Theorem \ref{tm:delaunay} shows that $f_{\eta,\pi/2} = f_\tau$ is proper.
\end{proof}

\section{Ideal polyhedra in hyperbolic ends}

\subsection{Hyperbolic ends}
\label{ssc:ends}

A $\bCP^1$-structure $\sigma$ on $S$ gives rise to a (non-complete) hyperbolic structure $E(\sig)$ on $S \times \bR_{>0}$, called the {\it hyperbolic end associated to} $\sig$.

The construction of this hyperbolic end is simpler in the special case where the developing map $\dev_\sigma:\tilde{S}\to \bCP^1$ is injective. In this case, the holonomy representation $\hol_\sigma$ of $\sigma$ acts properly discontinuously on the complement of the convex hull in $\bH^3$ of $\bCP^1\setminus \dev_\sigma(\tilde S)$, and the hyperbolic end $E(\sigma)$ can be defined as:
$$ E(\sigma) = (\bH^3\setminus \hull(\bCP^1\setminus \dev_\sigma(\tilde S)))/\hol_\sigma(\pi_1S)~. $$
The key geometric features of $E(\sigma)$ are visible in this special case. It is a hyperbolic manifold homeomorphic to $S\times \bR_{>0}$ with one complete end and one end whose metric completion is a pleated, concave surface homeomorphic to $S$.

In the general case, $E(\sig)$ can be constructed in the following manner. Start by associating a half-space $B_D$ in $\bH^3$ to each round disk $D$ in $\wt{\sigma}$. Recall that for a round disk $D$, $\dev_\sig(D)$ is a closed disk in $\bCP^1$, which we can identify with $\partial_\infty \bH^3$. The hyperplane in $\bH^3$ with boundary $\partial \dev_\sig(D)$ cuts $\bH^3$ into two half-spaces. The one facing $\dev_\sig(D)$ will be called $B_D$. We build the universal cover $\wt{E}(\sig)$ as a gluing of the half-spaces associated to all round disks in $\wt{\sig}$ by extending the natural overlap gluing between round disks. Finally, $E(\sig)$ is the quotient by the natural action of $\pi_1 S$. See \cite{kulkarni-pinkall} for a more detailed description of such constructions.

As in the special case, where $\dev_\sigma$ was injective, $E(\sigma)$ is a hyperbolic manifold homeomorphic to $S\times \bR_{>0}$. It is complete on one end, and has a metric completion whose boundary is a concave pleated surface homeomorphic to $S$. This concave surface is hyperbolic in the path metric and is pleated along a measured lamination. The complete end has an ideal (or conformal) boundary, which can be constructed via equivalence classes of rays with conformal structure arising from visual angles. To summarize and introduce notation,

\begin{definition}\label{def:ends}
Let $\sigma \in \cC$ be a complex projective structure on $S$. Define
\begin{enumerate}
\item $\partial_\infty E(\sigma)$ to be the the ideal boundary of $E(\sigma)$.
\item $\partial_0 E(\sigma)$ to be the concave, pleated surface on the metric boundary of $E$.
\item $h_\sigma \in \cT$ to be the induced hyperbolic metric on $\partial_0 E(\sigma)$
\item $\mu_\sigma \in \cML$ to be the measured bending lamination on $\partial_0 E(\sigma)$. 
\end{enumerate}  
\end{definition}

See \cite{thurston-notes,epstein-marden} for a detailed description of these objects.

\begin{remark}\label{rk:grafting}
  There is a close relationship between $\sig$, $h_\sigma$, and $\mu_\sigma$. Given $h_\sigma$, one can recover $\sig$ by {\it grafting} along $\mu_\sigma$. When $\mu_\sigma$ is a weighted multi-curve, the grafting operation is akin to cutting along each geodesic loop and attaching a flat cylinder of height given by the corresponding weight, for more details see \cite{dumas-survey, thurston-notes}. This operation extends to any hyperbolic  structure and any measured lamination giving a map  $\Gr :  \cT \times \cML \to \cC$, which turns out to be a homeomorphism \cite{kamishima-tan}. Composing with the projection $f : \cC \to \cT$, we get another grafting map $\gr : \cT \times \cML \to \cT$. Notice that $\Gr(h_\sigma, \mu_\sigma) = \sigma$. It is important to note that these statements are only known to be true when working on closed surfaces or, for complete surfaces of finite area, with measured laminations that do not enter the cusps. Extensions to hyperbolic surfaces with cone singularities of angle less than $\pi$ are also possible, see \cite{foliation}.
\end{remark}

\subsection{A review of measured laminations}

Let $(S_{g,n},h)$ be an oriented finite area hyperbolic surface of genus $g$ with $n \geq 0$ cusps. A \emph{geodesic lamination} $\cL$ on $(S_{g,n},h)$ is a closed subset which is the union of a disjoint collection of complete simple geodesics. A \emph{transverse measure}  $\mu$ on $\cL$ is an assignment of Radon measures $\mu_\al$ on each transverse arc $\al$ to $\cL$ subject to the following conditions: (1) if $\al' \subset \al$ is a sub-arc, then $\mu_{\al'} = \mu_\al\mid_{\al'}$ and (2) if $\al$ and $\al'$ are homotopic through a family of transverse arcs, then the homotopy sends $\mu_\al$ to $\mu_{\al'}$. The standard notation for the full measure of a transverse arc $\al$ is $i(\al, \mu)$.

Let $\cS$ denote the space of simple closed curves on $S_{g,n}$. A {\em weighted multi-curve} is a function $\omega : \cS \to \bR_{\geq 0}$ such that $\supp(\omega) = \{ \gamma \in \cS \mid \omega(\gamma) \neq 0\}$ is a finite collection of disjoint curves. Note that  $\supp(\omega)$ is a lamination. Viewing $\omega$ as assigning a weighted Dirac measure to transverse arcs gives an example of a transverse measure on $\supp(\omega)$.

Given a transverse measure $\mu$ on a geodesic lamination $\cL$, one can also define the geodesic lamination $\supp(\mu)$, which will always be a subset of $\cL$. A {\em measured lamination} is a pair $(\mu, \cL)$ such that $\supp(\mu) = \cL$. Because of this condition, we will usually only use $\mu$ when referring to a measured lamination. The space of measured laminations on $S_{g,n}$ will be denoted by $\cML = \cML_{g,n}$, with genus and punctures dropped when they are implied. The topology on $\cML$ is the weak topology on transverse measures. Note that changing the hyperbolic metric gives naturally homeomorphic spaces of measured laminations. See \cite{FLP} for a details.

Let $\cML^o$ denote the closure of weighted multi-curves in $\bR_{\geq 0}^\cS$. Thurston showed that by extending the correspondence between weighted multi-curves and measured laminations, we obtain $\cML^o$ as a subspace of $\cML$. Further, when there are no punctures, $\cML_{g,0} = \cML_{g,0}^o$.

For $S_{g, n}$ with $n \geq 1$, $\cML \smallsetminus \cML^o$ is non-empty and contains the measured laminations whose support includes bi-infinite simple geodesics with {\em both} ends going out a cusp. We call these {\em ideal leaves} and there can only be finitely many of them in any measured lamination, see \cite{thurston-notes}. An {\em ideal measured lamination} is one where {\em all} leaves are ideal. A simple example of an ideal measured lamination arises in the context of ideal hyperbolic polyhedra. If $P \subset \bH^3$ is an ideal polyhedron, then $\partial P$ is a pleated punctured sphere with the edges and dihedral angles of $P$ giving an ideal measured lamination on $\partial P$. That is, $\partial P$ is a hyperbolic punctured sphere bent along an ideal measured lamination.

\subsection{Ideal polyhedra}

In this section, we define a notion of ideal polyhedra in hyperbolic ends, and show that these ideal polyhedra are in one-to-one correspondence with Delaunay circle patterns. In the next section, we will show that ideal polyhedra in hyperbolic ends (with prescribed dihedral angles) are fully described by their induced metrics, and provide a description of these induced metrics.

An ideal polyhedron $P$ in $\bH^3$ is usually constructed as the convex hull of a finite set $V = \{v_1, \ldots, v_n\} \subset \partial_\infty \bH^3 = \bCP^1$, called \emph{ideal points}. More precisely, $P$ is the smallest convex set in $\bH^3$ such that $\partial_\infty P = V$. In the context of complex projective structures, we can think of $P$ as arising form the structure $\nu = \bCP^1 \smallsetminus V$ on the punctured sphere in the following way. Using our half-space construction from Section \ref{ssc:ends}, we can naturally consider the manifold $E(\nu)$ inside of $\bH^3$ and define $C =  \bH^3 \smallsetminus E(\nu)$.

We argue that $P = C$. Since $C$ is obtained cutting away half-spaces from $\bH^3$, it is convex. Further, $\partial_\infty C = V$ because we cut away all half-spaces bounding round disks in $\nu$. Lastly, if $C$ was not the smallest such convex set, there would be a hyperplane $Q$ cutting $C$ into non-empty $C^0 \cup C^1 = C$ such that, without loss of generality, $\partial_\infty C^0 = V$. However, this means $\partial_\infty Q$ bounds a round disk in $\nu$, contradicting the definition of $C$. 

One nice property of this perspective is that we automatically obtain a pleated structure on $\partial_0 E(\nu) = \partial P$ using Thurston's machinery \cite{thurston-notes, epstein-marden}. That is, $\partial P$ is the pleating of a hyperbolic metric on a punctured sphere along an ideal measured lamination.

We can now define an ideal polyhedron in a hyperbolic end. Fix $\sig \in \cC$, let $E = E(\sig)$ be the hyperbolic end, and let $\overline{E} = \partial_0 E \cup E$ be the metric completion of $E$. 

\begin{definition}
For a subset $\Omega \subset \partial_\infty E = \sig$, define 
\begin{enumerate}
\item $\hull_E(\Omega) = \overline{E} \smallsetminus E(\sig \smallsetminus \Omega)$.
\item $h_E(\Omega)$ to be the hyperbolic metric on $\partial_0  E(\sig \smallsetminus \Omega)$.
\item $\lambda_E(\Omega)$ to be the measured lamination on $\partial_0  E(\sig \smallsetminus \Omega)$.
\end{enumerate}
\end{definition}

When $V =  \{v_1, \ldots, v_n\}  \subset \partial_\infty E$ is a finite set of points, $\hull_E(V)$ is \emph{almost} an ideal polyhedron. Notice that topologically, $\partial_0  E(\sig \smallsetminus \Omega)$ will always be homeomorphic to $S_{g,n}$. However, the problem is that $\lambda_E(V)$ could be non-filling or a non-ideal measured lamination on $h_E(V)$. We step around this by adding this to be part of the definition.

\begin{definition} An \emph{ideal polyhedron} $P$ in a hyperbolic end $E$ is the set $P = \hull_E(V)$, where $V = \{v_1, \ldots, v_n\} \subset \partial_\infty E$, such that $\lambda_E(V)$ is a filling ideal measured lamination. For notational convenience, let $\partial P$ denote the ideal polyhedral surface $\partial_0  E(\sig \smallsetminus V)$.
\end{definition}

Given an ideal polyhedron $P = \hull_E(V)$, the {\em vertices} of $P$ are elements of $V$, the {\em edges} of $P$ are the leaves of $\lambda_E(V)$, and the {\em faces} of $P$ are the components of $\partial P \smallsetminus \lambda_E(V)$. Note that the faces of $P$ are totally geodesic and isometric to ideal hyperbolic polygons. This gives a polygon cell decomposition of $S$ associated to $P$.

We now turn to the close correspondence between Delaunay circle patterns in $(S,\sigma)$ and ideal polyhedra in $E(\sigma)$.

\begin{prop} \label{pr:correspondence}
  Let $P$ be an ideal polyhedron in $E(\sigma)$, where $\sigma\in \cC$. For each face $f$ of $P$, the outward normal at $f$ defines a unique immersed totally geodesic half-space $D_f \subset E(\sig)$ bounded by $f$, with $\partial_\infty B_f$ a round disk $D_f\subset \wt{\sig}$. The disks $D_f$, for $f$ in the set of faces of $P$, form a Delaunay circle pattern $\cD_P$. The polygon cell decomposition of $S$ associated to $P$ is isotopic to $\eta_\cD$. Further, $\theta_\cD(e)$ is equal to the exterior dihedral angle of $P$ at the corresponding edge.

 Conversely, given a Delaunay circle pattern $\cD$ in $(S,\sigma)$, each $D\in \cD$ is the boundary at infinity of an immersed half-space in $E(\sig)$. The intersection of the complement of those half-spaces is an ideal polyhedron $P\subset E(\sig)$. The polygon cell decomposition of $S$ associated to $P$ is isotopic to $\eta_\cD$. Further, for each edge $e$ of $P$, the exterior dihedral angle of $P$ at $e$ is equal to $\theta_\cD(e)$.
 
Further, when the combinatorics are fixed, this bijection is a homeomorphism.
\end{prop} 
\begin{proof} 

There is little to verify for the forward direction. For an ideal polyhedron $P$ in $E(\sig)$, the round disks $\cD = \{D_f\}$ cover $S$ and $V_\cD  = V$ is finite. For an edge $e$ edge of $P$,  let $f_e^0$ and $f_e^1$ be the two faces meeting at $e$. Then $D_{f_e^1}$ is a cutting disk for $D_{f_e^0}$, and vice versa. Since edges in $\eta_\cD$ correspond to cutting disk pairs, it follows that the cell decompositions agree. Further, the dihedral angle at $e$ is the angle between the outward normals of $B_{f_e^0}$ and $B_{f_e^0}$ measured along $e$, but this is the same as the angle between the corresponding round disks given as $\theta_\cD(e)$.

Going in the other direction, fix a Delaunay circle pattern $\cD$ on $\sig$ and let $V = V_\cD$. The main observation here is that the disks of $\cD$ are maximal. Indeed, by part (3) of the definition of Delaunay circle pattern, each round disk $D \in \cD$ has at least 3 elements of $V$ on $\partial D$. The associated half-space $B_D$ is therefore not properly contained in any union of half-spaced bounding round disks in $\sig \smallsetminus V$. By this maximality property and the fact that $\cD$ covers $S$, $P = E(\sig) \smallsetminus \bigcup_{D \in \cD} B_D$. It follows that edges of $P$ correspond to cutting disk pairs, the polygonal cell decompositions agree, and the dihedral angles are exactly given by $\theta_\cD$.

If we fix the admissible pair $(\eta,\theta)$, then any Gromov-Hausdorff converging sequence $P_n \subset E_n \to P \subset E$ of ideal polyhedra is hyperbolic ends with $(\eta, \theta)$ combinatorics will give a convergent sequence $(\sig_n, \cD_n) \in \cC_{\eta,\theta}$, and vice versa. Thus, the bijection is a homeomorphism.
\end{proof}

In light of this correspondence, we can define some notation. Given a Delaunay circle pattern $\cD$ on $(S,\sigma)$, let $P_\cD \subset E(\sig)$ be the associated ideal polyhedron, $\Sigma_\cD = \partial P_\cD = \partial_0  E(\sig \smallsetminus V_\cD)$ the corresponding ideal polyhedral surface, $h_\cD$ the hyperbolic metric on $\Sigma_\cD$ and $\lambda_\cD$ the measured lamination on $\Sigma_\cD$.

Notice that Proposition  \ref{pr:correspondence} implies that the dihedral angles of $P$ satisfy Lemma \ref{lem:angles}.

\subsection{Necessary conditions on angles}
\label{ssc:delaunay}

In this section we give a proof of Lemma \ref{lem:angles}.

\begin{proof}[Proof of Lemma \ref{lem:angles}]
  Let $\theta_\cD: \cE(\eta^*_\cD) \to (0,\pi)$ be the angle function associated to a Delaunay circle pattern $\cD$. Consider the polyhedron $P_\cD$ and polyhedral surface $\Sigma_\cD$ corresponding to $\cD$ defined in the previous section.

  Condition (1) is equivalent to enforcing that, for each ideal vertex $v$ of $P_\cD$, the sum of the dihedral angles of $P_\cD$ at $v$ is equal to $2\pi$. This follows directly from considering the link of $P_\cD$ at $v$ --- the intersection of $P_\cD$ with an embedded horosphere in $E(\sig)$ centered at $v$. By construction, this link is a Euclidean polygon with exterior angles equal to the dihedral angles of $P_\cD$ at $v$, so the angles sum to $2\pi$.

 To prove condition (2), let $\gamma$ be a homotopically trivial, non-backtracking closed edge path in $\eta^*_\cD$ which does not bound a face. We can lift $\gamma$ to a closed loop $\wt{\gamma} = [e_1, \ldots, e_k]$ in $\wt{\eta}_\cD^*$. Since the vertices of $\eta_\cD^*$ correspond to faces of $P_\cD$, $\wt{\gamma}$ defines a sequence $f_1, \ldots, f_k$ of ideal hyperbolic polygons. Topologically, the sequential edge gluing $T_\gamma = \bigcup_i  f_i$ is a cylinder which carries a hyperbolic metric. Further, $\partial T_\gamma$ is totally geodesic with cusps, so $T_\gamma$ is a convex surface and $\wt{\gamma}$ has a unique geodesic representative $g$ in $T_\gamma$. Since $\gamma$ was non-backtracking, $k \geq 3$.
  
We can realize $g$ as a piecewise geodesic loop on the lift of $\Sigma_\cD$ in $\wt{E}(\sig)$. Let $g_i$ be the image of the geodesic sub-arc of $g$ crossing $f_i$ and let $\al_i$ be the angle between $g_i$ and $g_{i+1}$. It follows that $\dev_{\wt{E}(\sig)}(g)$ is an immersed polygonal path in $\bH^3$ with exterior angles $\al_i$. Since $k \geq 3$ and the dihedral angles of $P_\cD$ are less than $\pi$, the vertices of $\dev_{\wt{E}(\sig)}(g)$ are not all collinear. Additionally, $\al_i \leq \theta_\cD(e_i)$ for all $i$, by construction.  By \cite[Theorem 3.1]{HR}, we have that $$2 \pi < \sum_i \al_i \leq \sum_i \theta_\cD(e_i),$$
which completes the proof.
\end{proof}

Note that a different, purely 2-dimensional argument can be found in \cite{bobenko-springborn} for hyperbolic or Euclidean circle patterns. However it does not seem to generalize easily to Delaunay circle patterns on surfaces equipped with general complex projective structures. 

\subsection{Balanced metrics and circle packings}

For $n\geq 1$, let $\cT_{g,n}$ denote the space of complete hyperbolic structures on $S_{g,n}$ of finite area, considered up to isotopy fixing the cusps. We will be interested in the setting where the cusps of these metrics correspond to vertices of polygonal cell decompositions. This context allows up to define the notion of ``cusp angles'' as follows.

\begin{definition}
Let $\eta$ be a polygonal cell decomposition of $S$ and let $h$ be a complete hyperbolic metric on $S$ with cusps at the vertices of $\eta$. Let $\cL$ be the geodesic realization of $\eta$ on $h$. Given a vertex $v \in \cV(\eta)$, let $e_1,\cdots, e_{k}$ be leaves of $\cL$ adjacent to $v$, in cyclic order. Fixing an embedded horocycle $c$ in $h$ centered at $v$, we define the {\em cusp angle} $\phi_i(v)$ to be the horocyclic length of the segment of $c$ between $e_i$ and $e_{i+1}$.
\end{definition}

Notice that replacing $c$ by another embedded horocycle centered at $v$ multiplies all cusp angles by the same number, so that the cusp angles $(\phi_i(v))_{i=1}^k$ are well-defined as elements of $(\bR_{>0})^k/\bR_{>0}$, where $\bR_{>0}$ acts on $(\bR_{>0})^k$ by diagonal multiplication.

\begin{definition}
Let $\kappa$ be a polygonal cell decomposition of $S$ with {\em all vertices of degree $4$} and fix $h\in \cT_{g,n}$ with cusps at the vertices of $\kappa$. We say that $h$ is {\em $\kappa$-balanced} if $\phi_1(v) = \phi_3(v)$ and $\phi_2(v) = \phi_4(v)$ for all $v \in \cV(\kappa)$. Let $\cB_S^\kappa \subset \cT_{g,n}$ denote the space of $\kappa$-balanced metrics with cusps at the vertices of $\kappa$.
\end{definition}

\begin{prop} \label{prop:balanced1}
 Let $P$ be a right angled ideal polyhedron in a hyperbolic end $E$, and let $h$ be the induced metric on $\partial P$. Then $h$ is $\kappa$-balanced, where $\kappa$ is the polygonal cell decomposition associated to $P$. 

Conversely, if $\kappa$ is a polygonal cell decomposition of $S$ with all vertices of degree $4$, each $\kappa$-balanced metric $h \in \cB_S^\kappa$ is induced by a unique pair $P \subset E$, where $P$ is a right angled ideal polyhedron and $E$ a hyperbolic end.

Further, when the combinatorics are fixed, this bijection is a homeomorphism.
\end{prop}

\begin{proof}
 Let $P$ be a right-angled ideal polyhedron in $E$, and let $h$ be the metric on $\partial P$. Given a face $f$ of $\kappa^*$, let $v_f$ be the corresponding ideal vertex of $P$. The link of $P$ at $v_f$ --- the intersection of $P$ with an embedded horosphere in $E$ centered at $v_f$ --- is a Euclidean polygon with right angles, that is, a rectangle. Therefore, its opposite edges have equal length, which means precisely that $h$ is $\kappa$-balanced at $v_f$.

Conversely, fix $h \in \cB_S^\kappa$ and let $\cL$ be the geodesic realization of $\kappa$ on $h$. Fix an embedded horocycle $c_i$ for each cusp of $h$. Cutting the surface along $\cL$ gives a disjoint union of ideal hyperbolic polygons $F_j$, where each ideal corner is marked by a sub-arc of one the $c_i$'s. We can build a hyperbolic end for $S_{g,n}$ as follows. Realize each ideal polygon $F_j$ on the boundary of some half space $B_j$. If two faces  share an ideal edge, we glue the two associated half spaces along that edge at an angle of $\pi/2$ such that the end points of the marked sub-arcs meet. This gluing gives a non-complete hyperbolic structure on $S_{g,n} \times \bR>0$. The balancing condition guarantees that the gluing around a cusp has trivial holonomy and, therefore, the conformal structure at infinity can be ``filled'' at each cusp to a complex projective structure $\sig_h$ on $S$. Let $V$ be the set of filled points on $\sig_h$, then $E = E(\sig_h)$ and $P = \hull_E(V)$. By construction, the metric of $\partial P$ is exactly $h = h_E(V)$.

If we fix the cell decomposition $\kappa$, then any Gromov-Hausdorff converging sequence $P_n \subset E_n \to P \subset E$ of right angled ideal polyhedra is hyperbolic ends with cell decomposition $\kappa$ will give a convergent sequence of metrics on $\partial P_n$, and vice versa.
\end{proof}

Next, we give a corollary of Proposition \ref{prop:balanced1} in the context of our setup. This corollary was known to Kojima, Mizushima and Tan \cite[Proposition 3.1]{KMT}.

\begin{cor}\label{cor:variety1} Let $\tau$ be a polygon cell decomposition of $S$. Then $C_\tau$ is a real (semi)algebraic set of formal dimension $6g-6$.
\end{cor}

\begin{proof} By Propositions \ref{prop:balanced1} and  \ref{pr:correspondence}, $C_\tau$ is homeomorphic to $\cB_S^\kappa$, where $\kappa$ is the midpoint cell decomposition from Remark \ref{rk:nerves}. Notice that $\cB_S^\kappa \subset \cT_{g,n}$ where $n = \# \cV(\kappa)$. For every vertex $v$ of $\kappa$, we have the cusp angle relations $\phi_1(v) = \phi_3(v)$ and $\phi_2(v) = \phi_4(v)$. If we complete the 1-skeleton of $\kappa$ to a maximal ideal lamination $\lambda$ (one where all complementary regions are triangles), then we can parametrize $\cT_{g,n}$ by {\em multiplicative shearing coordinates} with respect to $\lambda$.

Briefly, these coordinates assign to each leaf $\ell$ of $\lambda$ a number $s_\ell \in \bR_{>0}$ and are defined as follows. For $h \in \cT_{g,n}$, realize $\lambda$ as a lamination and draw the the incircle of each complementary ideal triangle. Then, for every leaf $\ell$ of $\lambda$, take the signed right-hand distance $d_\ell$ between the tangency points of the two incircles on either side. Define $s_\ell = \exp(d_\ell)$. The Euler characteristic relation tells us there are $6g-6+3n$ leaves of $\lambda$, so this defines a map $T_{g,n} \to \bR_{>0}^{6g-6+3n}$. The image corresponds to points where for each vertex $v$ of $\lambda$, the product of all shearing coordinates around the vertex is $1$. This guarantees that horocycles around $v$ ``close up'' to form a cusp. Indeed, notice that if one takes an ideal triangle with a horocyclic arc of length $l$ cutting off a cusp, then continuing that arc into a neighboring triangle with shear $s_\ell$ with give a horocyclic arc of length $s_\ell \cdot l$. To get a complete cusp, we must have the full product off all the shears around a cusp equal to $1$. The coordinate map will be a homeomorphism onto its image, a ball of dimension $6g-6+2n$. See  \cite{bonahon-toulouse} for more details.

The relations $\phi_1(v) = \phi_3(v)$ and $\phi_2(v) = \phi_4(v)$, which measure lengths of horocyclic arcs, become polynomial relations in the multiplicative shearing coordinates, giving us that $\cB_S^\kappa$ is a real  algebraic set of formal dimension $6g-6$.
\end{proof}

\begin{remark}
  If we could show that $\cB_S^\kappa$ is irreducible, then it would follow that $C_\tau$ has dimension exactly $6g-6$ at the smooth points. Kojima, Mizushima and Tan have shown that $C_\tau$ admits a neighborhood of the Koebe-Andreev-Thuston solution that is homeomorphic to a ball of dimension $6g-6$, see \cite[Lemma 3.2]{KMT}.

  In fact a ``double doubling'' construction in \cite{KMT} shows that $C_\tau$ is smooth of dimension $6g-6$ at all points $(\sigma,\cP)$ where $\sigma$ is quasi-Fuchsian, that is, when $\dev_\sig$ is injective. Indeed, at such a point, one can consider the manifold $$M_1(\sigma,\cP) = (\bH^3 \smallsetminus \wt{E}(\sig \smallsetminus V_{\cD_\cP}))/\pi_1 S,$$ which is a (non-complete) hyperbolic manifold with polyhedral boundary, with right angles at all its edges and a connected ideal boundary. One can glue two copies of $M_1(\sigma,\cP)$ along the faces corresponding to the dual disks, to obtain a hyperbolic manifold $M_2(\sigma,\cP))$ with totally geodesic boundary. Then, we can glue two copies of $M_2(\sigma,\cP))$ along their boundaries to obtain a complete hyperbolic manifold $M_4(\sigma)$ with an ideal boundary composed of cusps and four surfaces of genus $g$. This manifold admits an isometric action of $\Z/2\Z\times \Z/2\Z$ permuting the ideal boundary surfaces. Moreover, any small deformation of $\wt{M}_4(\sigma,\cP)$ invariant under the action of $\Z/2\Z\times \Z/2\Z$ arises from a small deformation of $(\sigma, \cP)$. The Ahlfors-Bers Theorem therefore shows that near $(\sigma,\cP)$, $C_\tau$ is smooth of dimension $6g-6$.
\end{remark}

\subsection{Metrics associated to Delaunay circle patterns}

Just as the balanced metrics are precisely the metrics associated to circle packings, there is a clear description of the metrics associated to Delaunay circle patterns.

\begin{definition}
  Let $(\eta, \theta)$ be an admissible pair for $S$ and fix $h\in \cT_{g,n}$ with cusps at the vertices of $\eta$. We say that $h$ is {\em $(\eta, \theta)$-balanced} if for every face $f$ in $\eta^*$ with $\partial f = [e_1, \ldots, e_k]$, the sequences $(\phi_1(v_f), \ldots, \phi_k(v_f))$ and $(\theta(e_1), \ldots, \theta(e_k))$ are the edge lengths and exterior angles of a Euclidean polygon, respectively. Here, $v_f$ is the dual vertex in $\eta$ and $\phi_i$ are the cusp angles between $e_i^*$ and $e_{i+1}^*$. This condition can be written explicitly as
  \begin{enumerate}
\item $\theta(e_1)+\cdots +\theta(e_n) = 2\pi$ and
\item $\phi_1(v_f)+ \phi_2(v_f) e^{i\theta(e_1)} + \phi_3(v_f) e^{i(\theta(e_1)+\theta(e_2))} + \cdots + \phi_n(v_f)e^{i(\theta(e_1)+\cdots+\theta(e_{n-1}))} = 0.$
\end{enumerate}

 Notice that (1) always holds as $(\eta,\theta)$ is an admissible pair. Let $\cB_S^{\eta,\theta} \subset \cT_{g,n}$ denote the space of $(\eta,\theta)$-balanced metrics with cusps at the vertices of $\eta$.
\end{definition}

Note that $(\phi_1(v_f), \ldots, \phi_k(v_f))$ are well-defined only up to multiplication of all the terms by a positive number. However this transformation does not change whether the cusp angles are edge lengths of a Euclidean polygon with exterior angles $(\theta(e_1), \ldots, \theta(e_k))$ .

\begin{prop} \label{prop:balanced2}
 Let $P$ be an ideal polyhedron in a hyperbolic end $E$ and let $h$ be the induced metric on $\partial P$. Let $\eta$ be the polygonal cell decomposition of $P$ and let $\theta : \cE(\eta^*) \to (0,\pi)$ the exterior angle function, then $h$ is $(\eta, \theta)$-balanced.
 
 Conversely, if $(\eta, \theta)$ is an admissible pair for $S$, then each metric $h \in \cB_S^{\eta,\theta}$ is induced by a unique pair $P \subset E$, where $P$ is an ideal polyhedron and $E$ a hyperbolic end.
 
Further, when the combinatorics are fixed, this bijection is a homeomorphism.
\end{prop}

\begin{proof} The proof is a direct generalization of the argument in Proposition \ref{prop:balanced1}.
\end{proof}

\begin{cor} Let $(\eta, \theta)$ be an admissible pair for $S$. Then $C_{\eta, \theta}$ is a real (semi)algebraic set of formal dimension $6g-6$.
\end{cor}

\begin{proof} As in the proof of Corollary \ref{cor:variety1}, Propositions \ref{prop:balanced2} and  \ref{pr:correspondence} give that $C_{\eta,\theta}$ is homeomorphic to $\cB_S^{\eta,\theta}$. We can similarly complete $\eta$ to a maximal ideal lamination $\lambda$ and define shearing coordinates on $T_{g,n}$. Then, $\cB_S^{\eta,\theta}$ is a real algebraic set in $T_{g,n}$ defined by the conditions that for every face $f$ in $\eta^*$ with $\partial f = [e_1, \ldots, e_k]$, the sequences $(\phi_1(v_f), \ldots, \phi_k(v_f))$ and $(\theta(e_1), \ldots, \theta(e_k))$ are the edge lengths and exterior angles of a Euclidean polygon, respectively. Since $\theta$ is fixed, these give co-dimension $2$ polynomial equations in the shearing coordinates at every vertex of $\eta$ and therefore $\cB_S^{\eta,\theta}$ has formal dimension $6g-6$.
\end{proof}

\begin{remark}
As for circle packings, it can be expected that $C_{\eta,\theta}$ is smooth of dimension $6g-6$ at least at points $(\sigma, \cD)$ where $\sigma$ is quasi-Fuchsian. Given such a point $(\sigma,\cD)$, one can again consider the manifold $M_1(\sigma,\cD) = (\bH^3 \smallsetminus \wt{E}(\sig \smallsetminus V_\cD))/\pi_1 S$. Gluing two copies of $M_1(\sigma,\cD)$ along their geodesic boundaries yields a complete hyperbolic cone-manifold $M_2(\sigma, \cD)$ with cone singularities of angle less than $2\pi$ and an ideal boundary composed of cusps and two surfaces of genus $g$. This time, $M_2(\sigma, \cD)$ admits an isometric involution. Moreover, any small deformation of $M_2(\sigma, \cD)$ invariant under this involution comes from a small deformation of $(\sigma,\cD)$. The smoothness of $C_{\eta,\theta}$ at $(\sigma,\cD)$ would therefore follow from a local rigidity statement for $M_2(\sigma, \cD)$. Rigidity statements of this type are known for ``convex co-compact'' hyperbolic cone-manifolds with cone singularities along bi-infinite geodesics, see \cite{qfmp}, but do not quite cover the case of $M_2(\sigma, \cD)$ with singularities along bi-infinite geodesics going between cusps.
\end{remark}

\section{Proof of Theorem \ref{tm:delaunay}}

\subsection{A uniform bound on the pleating lamination}\label{ssc:pleating}

In Definition \ref{def:ends}, we defined the map $\sig \mapsto \mu_\sig$ sending a complex projective structure $\sigma$ to the pleating measured lamination $\mu_\sig$ of $\partial_0 E(\sigma)$. Call this map $L:\cC\to \cML$. The following Lemma is an extension of \cite[Lemma 4.1]{KMT2} controlling the image of $L$.

\begin{lemma} \label{lem:lambda}
Let $(\eta, \theta)$ be an admissible pair for $S$. Then $L\circ \Pi_1(\cC_{\eta,\theta})$ is pre-compact in $\cML$.
\end{lemma}

\begin{proof}
To show that $L\circ \Pi_1(\cC_{\eta,\theta})$ is pre-compact in $\cML$ it is enough to show that there exists a filling closed curve $\gamma$ and a constant $K$ such that for any $\sigma \in \cC_{\eta,\theta}$, $i\left(\gamma, \mu_\sigma\right) < K$, see \cite[Prop. 4]{bonahon_geometry}. Fix any non-backtracking filling closed curve $\gamma$ running along the edges of $\eta^*$ and let $N$ be the number of edges of $\eta^*$ that appear in $\gamma$.

Fix $\sig \in \cC_{\eta,\theta}$ and let $E = E(\sig)$. Working in the universal cover, define the {\it retraction map} $\wt{r} : \wt{\sig}  \to \partial_0 \wt{E}$ as follows. For every $p \in  \wt{\sig}$, let $\wt{r}(p)$ be the tangency point of the maximal horoball $H_p \subset \wt{E}$ centered at $p$ to $\partial_0 \wt{E}$. By the concavity of $\partial_0 \wt{E}$, $\wt{r}$ is well defined.

For every point $p \in \wt{\sigma}$, there is a {\it support half-space} $F_p \subset \wt{E}$ containing $H_p$ which is cut off by the hyperplane tangent to $H_p$ at $\wt{r}(p)$. Notice that $F_p$ is M\"obius equivalent of a half-space of $\bH^3$ and $\partial_\infty F_p$ is a round disk in $\wt{\sigma}$. Further, the assignment $p \mapsto F_p$ is continuous. For details, see \cite{epstein-marden}.

Let $\cD$ be the Delaunay circle pattern realized on $\sigma$ corresponding to $\eta$ and $\theta$. We will realize $\wt{\eta}^*$ on $\wt{\sig}$ by first choosing the location of vertices.

{\it Claim:} For every $D \in \wt{\cD}$ there is a point $v_D \in \mathrm{int}(D) \subset \wt{\sig}$ such that $D \subset \partial_\infty F_{v_D}$.

{\it Proof of Claim:} Let $B_D \subset \wt{E}$ be the half-space defined by $D$. Pick a point $p \in \partial B_D$ minimizing the distance from $\partial B_D$ to $\partial_0 \wt{E}$ and let $\al$ be the geodesic segment realizing this distance. Let $q \in \partial_0 \wt{E}$ be the other endpoint of $\al$ and let $F_q$ be the half-space perpendicular to $\al$ at $q$. Since $\partial_0 \wt{E}$ is concave, we can guarantee that $F_q$ is an embedded half-space in $\wt{E} \cup \partial_0 \wt{E}$, giving us that $B_D \subset F_q$. Exponentiating $\al$ gives an ideal point $v_D$. Because $B_D \subset F_q$ and $\al$ is perpendicular to both, an expanding horoball at $v_D$ will first become tangent to $\partial_0 \wt{E}$ at $q$, implying that $F_{v_D} = F_q$ and  $D \subset \partial_\infty F_{v_D}$.\qed

We can now equivariantly realize each vertex of $\wt{\eta}^*$ by choosing the point $v_D$ for the corresponding disk. The goal of this construction is to ensure that if the vertices $u,v$ are joined by an edge in $\wt{\eta}^*$, then $F_u$ and $F_v$ intersect. Indeed, this follows because $D_u$ and $D_v$ overlap, $D_v \subset \partial_\infty F_v$, and $D_u \subset \partial_\infty F_u$.

To realize an edge of $\wt{\eta}^*$, take the two endpoints $u, v$ and consider the geodesic $\al_{u,v}$ on $\partial_0 \wt{E}$ between $\wt{r}(u)$ and $\wt{r}(v)$. Take the edge to be $\wt{r}^{-1}(\al_{u,v})$. Note that if $\al_{u,v}$ runs along a leaf of $\mu_\sig$, we can slightly move $u$ and $v$ to make $\al_{u,v}$ transverse, while keeping the intersection between $F_v$ and $F_u$.

Returning to our path $\gamma$, we want to show that $i\left(\gamma, \mu_\sigma\right) < K$, for a constant $K$ independent of $\sig$. Consider a lift $\wt{\gamma}$ as an edge path in $\wt{\eta}^*$ realized on $\wt{\sig}$. The sequence of vertices gives a sequence $F_i$ of support half-spaces forming {\em roofs} over $\wt{r}(\wt{\gamma})$, see Figure \ref{fig:roof_simple}. Since $\wt{r}(\wt{\gamma})$ is a piecewise geodesic on $\partial_0 \wt{E}$ transverse to $\mu_\sigma$, the Roof Lemma of Bridgeman-Canary \cite[Lemma 4.1]{bridgeman-canary} applies and tells us that the total bending $i\left(\gamma, \mu_\sigma\right)$ is less than $N \pi$.

\begin{figure}[htb]
\begin{center}
\begin{overpic}[scale=.7]{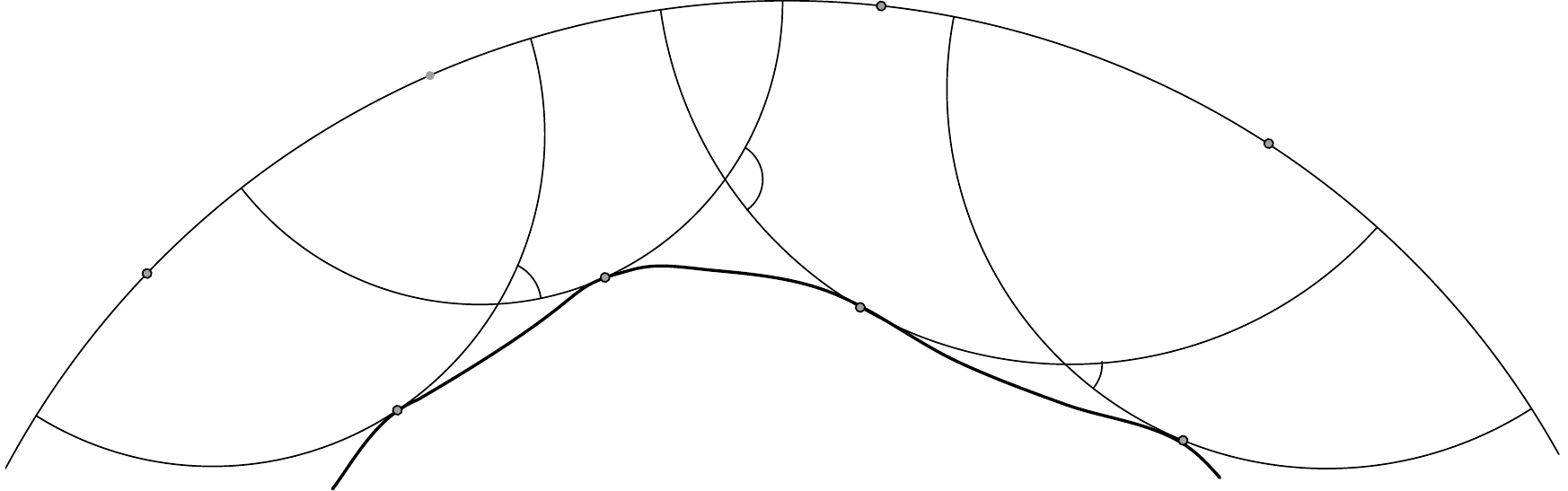}
\put(6.5,15.5){$v_0$}
\put(25,28){$v_1$}
\put(57,32){$v_2$}
\put(81,23){$v_3$}
\put(25,2){$\wt{r}(v_0)$}
\put(38,10.5){$\wt{r}(v_1)$}
\put(51,8.5){$\wt{r}(v_2)$}
\put(70,0.5){$\wt{r}(v_3)$}
\put(13,3){$F_0$}
\put(23,15){$F_1$}
\put(78,13){$F_2$}
\put(89,4){$F_3$}
\put(34.2,14){$\theta_1$}
\put(49.5,19){$\theta_2$}
\put(71,6){$\theta_3$}
\end{overpic}
\end{center}
\caption{The half-spaces form a sequence of {\em roofs} over $\wt{r}(\wt{\gamma})$. The Roof Lemma tells us that the bending measure of the arc between $\wt{r}(v_0)$ and $\wt{r}(v_3)$ is less than $\sum_i \theta_i < 3 \pi$.}
\label{fig:roof_simple}
\end{figure}
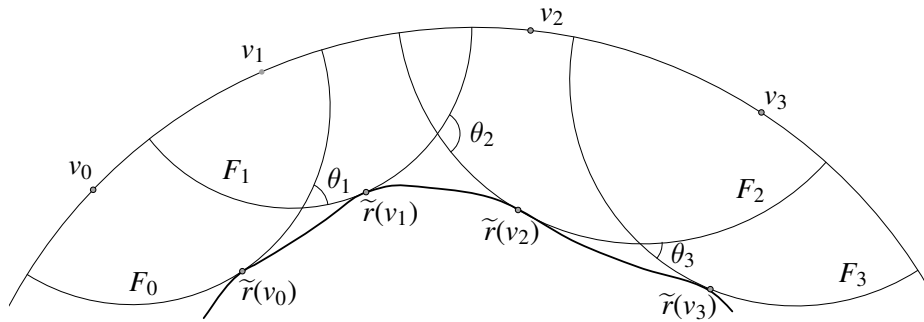

Since $\sig \in \bC_{\eta, \theta}$ was arbitrary and $K = N \pi$ only depends on the combinatorics of $\eta$, we can conclude that $L\circ \Pi_1(\cC_{\eta,\theta})$ is pre-compact in $\cML$.\end{proof}

\subsection{Bounds on the induced metrics}

Given $(\sigma, \cD)\in \cC_{\eta,\theta}$, let $\Sigma_{\cD}\subset E(\sigma)$ be the corresponding polyhedral surface. A closed geodesic on $\Sigma_{\cD}$ is a closed geodesic on the corresponding balanced metric in $\cB_S^{\eta,\theta}$. In this section, we show that the lengths of these geodesics can be bounded below in terms of $(\eta,\theta)$.

\begin{prop} \label{prop:contractible}
Let $(\eta, \theta)$ be an admissible pair for $S$. There is a constant $L_0>0$ depending only on $(\eta, \theta)$ such that if $(\sigma,\cD)\in \cC_{\eta,\theta}$ and $\gamma$ is a closed geodesic on $\Sigma_{\cD}$ which is contractible in $S$, then the length of $\gamma$ is at least $L_0$.
\end{prop}

\begin{proof} As the geodesics of interest are contractible in $S$, we can always lift them to loops in $\wt{E} = \wt{E}(\sig)$. Fix a lift $\Sigma_{\wt{\cD}}$ of $\Sigma_{\cD}$ in $\wt{E}$ and consider a closed geodesic $\gamma \subset \Sigma_{\wt{\cD}}$ of length $L > 0$. Note that $\gamma$ is a piecewise geodesic in $\wt{E}$. We will show that $L$ cannot be smaller than some constant $L_0>0$ depending only on $(\eta, \theta)$. Since we are giving a lower bound on length, it is safe to assume that $\gamma$ is simple.

Let $\{F_i\}_{i = 1}^n$ be the cycle of support half-spaces to $\Sigma_{\wt{\cD}}$ along the geodesic arcs of $\gamma$ and let $e_i = \partial_0 F_i \cap \partial_0 F_{i+1}$ be the pleating lines along $\gamma$. The convex hull of $e_i \cup e_{i+1}$ for each sequential pair is an ideal triangle or quadrilateral in $\partial_0 F_i$. Let $T \subset \Sigma_{\wt{\cD}}$ be the hyperbolic cylinder obtained by cyclically gluing these ideal triangles and quadrilaterals along the $\{e_i\}$. See Figure \ref{fig:T}. Intrinsically, $T$ has cusped, totally geodesic boundary and embeds in $\wt{E}$ by pleating along $\{e_i\}$. Since $\gamma$ is the core geodesic of $T$ and has non-zero length, the edges $\{e_i\}$ do not all share an ideal point. Our proof will proceed in two steps: $(i)$ when $L$ is small, $\left.\dev_{\wt{E}}\right|_T$ is an embedding $(ii)$ such an embedding can only exist if $L > L_0$.

\begin{figure}[htb]
\begin{center}
\begin{overpic}[scale=.7]{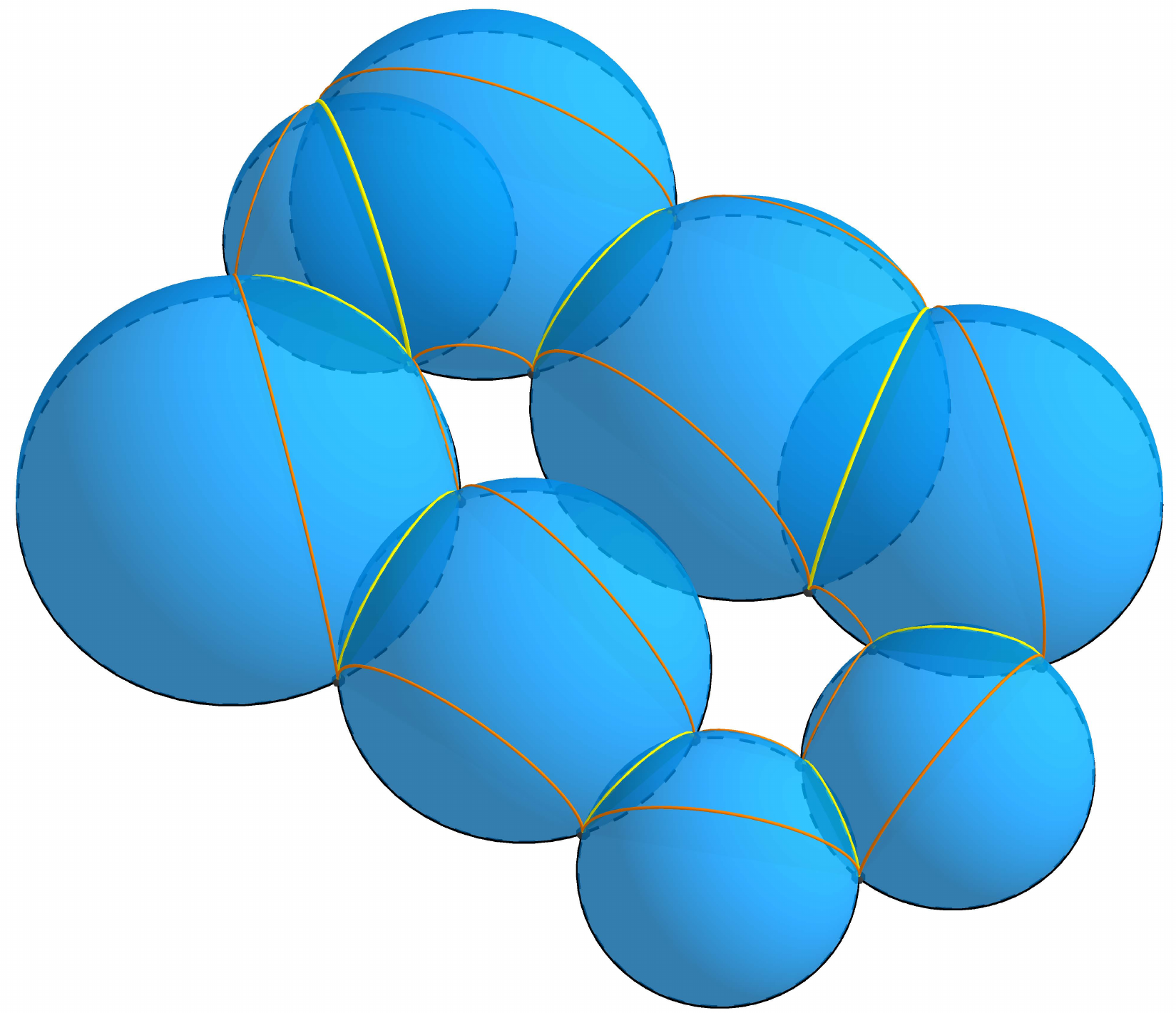}
\end{overpic}
\end{center}
\caption{Image of $\{F_i\}$ and $T$ in the upper half-space model of $\bH^3$. The yellow geodesics are the pleating lines along $\gamma$ and the orange geodesics form $\partial T$. The length $L$ of $\gamma$ is large in this picture, but keep in mind $F_i$ and $F_j$ may intersect for $|i - j| > 1$ even if $L$ is small.}
\label{fig:T}
\end{figure}

{\it Claim 1:} When $L$ is small, $\left.\dev_{\wt{E}}\right|_T$ is an embedding.

{\it Proof of Claim 1:} To prove this Claim, we will find an isometrically embedded hyperbolic ball $B_\gamma \subset \wt{E}$ such that $\gamma \subset B_\gamma$ when $L$ is small enough. Since balls are convex, $\left.\dev_{\wt{E}}\right|_{B_\gamma}$ will be an embedding and, in particular the surface $B_\gamma \cap T$ will embed in $\bH^3$. Since the $\{e_i\}$ are disjoint in $\wt{E}$, the images of the segments $e_i \cap B_\gamma$ will exponentiate to disjoint geodesics in $\bH^3$. The embedding of the pleating lines extends to that of $T$. It remains to find this ball $B_\gamma$.

Notice that a ball $B_{\wt{E}}(x,r)$ around $x \in \wt{E}$ of radius $r$ is isometric to a hyperbolic ball of radius $r$ if and only if the distance from $x$ to $\partial_0 \wt{E}$ to be greater than $r$. Pick a point $x \in \gamma$ and define $B_\gamma = B_{\wt{E}}(x, L)$. It suffices to prove that there exists an $L$ such that the distance from $x$ to $\partial_0 \wt{E}$ is greater than $L$.

We will use the visual perspective from $x \in \gamma$. Let $\bS_x$ be the unit tangent sphere at $x$ and define the visual projection $\pi_x : \wt{\sig} \to \bS_x$ as follows. For $z \in \wt{\sig}$, there is a unique ray $\rho_z : [0, \infty] \to \wt{E} \cup \partial_0 \wt{E}$ from $x = \rho(0)$ to $z = \rho_z (\infty)$ such that $\rho_z ((0,t))$ is the shortest path from $x$ to $\rho_z (t)$ for all $t$. We define $\pi_x(z)$ to be the unit vector tangent to $\rho_z$ at $x$. Notice that $\rho_z$ can either be totally geodesic or it can run along $\partial_0 \wt{E}$. Define the ``visible set'' $V_x  = \{ z \in \wt{\sig} \mid \rho_z \text{ is totally geodesic}\}$. Notice that since $\partial_0 \wt{E}$ is concave, $V_x$ is a closed topological disk.

To get a flavor for the structure of $V_x$, we will show that $\bS_x \smallsetminus \pi_x(V_x)$ is convex in the spherical metric. Note that for $z \in \partial V_x$, either $\rho_z$ tangentially meets $\partial_0 \wt{E}$ at some point $w$ or $z$ lies in the limit set of $\wt{\sigma}$. In the former case, there is a pleating line of $\partial_0 \wt{E}$ through $w$, which along with $\rho_z$ spans a plane that cuts off a half-space $F$ in $\wt{E}$. By convexity, $F$ is embedded and $\partial_\infty F \subset V_x$. It follows that the projection $\pi_x(\partial_\infty F)$ is a hemisphere avoiding and tangent to $\bS_x \smallsetminus \pi_x(V_x)$ at $\pi_x(z)$. If $\rho$ does not intersect $\partial_0 \wt{E}$, then $z$ is on the limit set of $\wt{\sig}$ and there is some support half-space of $\partial_0 \wt{E}$ tangent at $z$ that must contain $\rho_z$. Inside, we can take a smaller half-space $F$ with $\rho_z \subset \partial_0 F$ and proceed as before. This proves that $\bS_x \smallsetminus \pi_x(V_x)$ is convex.

Let $C = \bigcup_i F_i$ and notice that $\partial_\infty C \subset \wt{\sig}$ is a multiply-connected domain. Let $\partial_\infty C =  R_0 \smallsetminus \bigcup_j \mathrm{int}(R_j)$ for closed topological disks $R_j \subset R_0 \subset \wt{\sigma}$. If $R_0 \subset V_x$, then we can embed $C$, and therefore $T$, in $\bH^3$ by identifying $\bS_x$ with the visual sphere of any point in $\bH^3$ and exponentiating the set of vectors $\pi_x(R_0 \cap V_x)$. Thus, we may assume that $R_0 \not \subset V_x$.

Let $r = d_{\wt{E}}(x,\partial_0 \wt{E})$ and let $y \in \partial_0 \wt{E}$ realize this distance. Consider the closest complete line $\ell$ of $\partial_0 \wt{E}$ to $y$. Notice that $\ell$ does not necessarily need to be a pleating line. Since the totally geodesic components of $\partial_0 \wt{E}$ are always ideal hyperbolic polygons, the worst case scenario is that $y$ lies in the center of an ideal triangle. Therefore, the distance from $\ell$ to $y$ is bounded above by $\log(3/2)$. By convexity, the endpoints $\{p,q\} = \partial_\infty \ell$ lie in $V_x$ and $\{x,p,q\}$ defines an embedded totally geodesic hyperbolic triangle in $\wt{E}$. At $x$, this triangle has a non-zero angle $\phi$, which is equal to the spherical distance between $\pi_x(p)$ and $\pi_x(q)$. Classical hyperbolic geometry tell us that the distance from $x$ to $\ell$ is $\mathrm{arccosh}(1/\sin(\phi/2))$, see \cite{beardon}. Let $y' \in \ell$ be the closest point of $\ell$ to $y$, then $\{x,y,y'\}$ is either a right-angled hyperbolic triangle or $y = y'$. It follows that $\cosh\left(d_{\wt{E}}(x, y')\right) \leq \cosh(r) \cosh(\log(3/2)) = 13 \cosh(r)/12$. In particular, $r$ is bounded below by $\mathrm{arccosh}\left(12/(13 \sin(\phi/2))\right)$, which implies that $r \to \infty$ as $\phi \to 0$. We will show that $\phi \to 0$ as $L \to 0$, guaranteeing that there is some $L$ for which $r > L$.

Since $\ell$ is a pleating geodesic of $\partial_0 \wt{E}$, the endpoints $\{p,q\}$ lie on the limit set of $\wt{\sig}$, and therefore $p,q \notin \mathrm{int}(R_0)$. It follows that $p,q \in \partial V_x \smallsetminus \mathrm{int}(R_0)$ and it will be sufficient to show that the (spherical) diameter of $\pi_x(V_x \smallsetminus \mathrm{int}(R_0))$ goes to $0$ as $L \to 0$.

Notice that $\pi_x(V_x \smallsetminus \mathrm{int}(R_0))$ is topologically a union of closed disks $\{D_i\}$ where the boundary of each disk is composed of two arcs: one from $\partial R_0$ and the other from $\partial V_x$. Let $\partial D_i = \al_i \cup \beta_i$ with $\al_i \subset \pi_x(\partial R_0)$ and $\beta_i \subset \pi_x(\partial V_x)$. Since $\bS_x \smallsetminus \pi_x(V_x)$ is convex, each $D_i  \subset V_x $ will be contained in the (spherical) convex hull of $\al_i$ when $\diam \al_i < \pi/2$. Thus, we must show that $\diam \bigcup_i \al_i$ goes to $0$ as $L \to 0$.

For $z_1 \in \al_i$ and $z_2 \in a_j$, let $\delta$ be an arc along $\partial R_0$ connecting $z_1$ to $z_2$. Let $\{w_k\}_{k = 1}^m$ be the endpoints of the pleating lines of $C$ along $\delta$ and set $w_0 = z_1$, $w_{m+1} = z_2$. We have that $$d_{\bS_x}(\pi_x(z_1), \pi_x(z_2)) \leq \sum_{k =  0}^m d_{\bS_x}(\pi_x(w_k), \pi_x(w_{k+1})).$$
Further, this holds for all $x \in \gamma$. Fix $k$ and let $x$ vary along $\gamma$. Then, $d_k = d_{\bS_x}(\pi_x(w_k), \pi_x(w_{k+1}))$ is maximized when $x$ is closest to the complete geodesic connecting $w_k$ to $w_{k+1}$. Call this point $x_k \in \gamma$. Notice that $x_k$ must lie on the flat of $\partial_0 C$ between the two pleating geodesics terminating at the points $w_k$ and $w_{k+1}$, respectively.  This gives a totally geodesic triangle with vertices $\{x_k, w_k, w_{k+1}\}$ embedded in $\partial_0 C$. Classical hyperbolic geometry in the plane gives us a bound $d_k \leq 2 \arccos(\sech(L/2))$ and therefore $d_{\bS_x}(\pi_x(z_1), \pi_x(z_2)) \leq 2 (N+1) \arccos(\sech(L/2))$, where $N$ is the number of edges in $\eta$. Taking the supremum over all $z_1 \in \al_i$, $z_2 \in a_j$ and $i,j$ shows that $\diam \bigcup_i \al_i \leq 2 (N+1) \arccos(\sech(L/2))$, as desired.\qed

Now that we have reduced the problem to $\bH^3$, we can restate our task in terms of arrangements of half-spaces. Let $\mathcal{F}$ denote the space of half-spaces in $\bH^3$. Recall that  $\mathcal{F}$ carries a Lorentzian metric obtained by realizing $\mathcal{F}$ as the de Sitter space for the hyperboloid model of $\bH^3$. A key property of this metric is that two half-space $F_a, F_b \in  \mathcal{F}$ are joined by a space-like geodesic of length $\psi$ if and only if they intersect in $\bH^3$ and have $\psi$ as the angle between their outer normals in $\bH^3$. See \cite{Kulkarni:1994ik} for details.

Let $\hat{\theta}_{\gamma} = (\theta_1, \ldots, \theta_n)$ denote the sequence of pleating angles on $T$. Notice that $T$ defines a space-like polygonal loop in $\mathcal{F}$ with vertices $F_1, \ldots, F_n$ and edges of length $\theta_i$. Our approach will be to work with the parameter space of all space-like loops that could arise from $T$.

Given a space-like polygonal loop $P = \{P_i\} \in \mathcal{F}^n$, we can define the pleating edges $e_i(P) = \partial_0 P_i \cap \partial_0 P_{i+1}$ and the corresponding pleating angles $\theta_i(P)$. We say that $P$ is {\em simple} if (1) the $e_i$ are pairwise disjoint, (2) the $e_i$ do not {\em all} share an ideal point, and (3) the corresponding pleated cylinder $T_P$ is embedded. Recall that $T_P$ is built by cyclically gluing the ideal triangles or quadrilaterals given as convex hulls of $e_i \cup e_{i+1}$ in $\partial_0 P_i$. We use $\gamma_P$ to denote the core geodesic of $T_P$.

Let $\mathcal{A}_{\hat{\theta}_{\gamma}}$ denote the space of all simple space-like polygonal loops in $\mathcal{F}$ with pleating angles given by $\hat{\theta}_{\gamma}$. Then, let $\mathcal{A}_{\hat{\theta}_{\gamma}}(L)$ be the set of all $P \in \mathcal{A}_{\hat{\theta}_{\gamma}}$ where the hyperbolic length of $\gamma_P$ is less than $L$. Our goal is to show the following.

{\it Claim 2:} There exists $L_0> 0$ depending only $(\eta,\theta)$ such that $\mathcal{A}_{\hat{\theta}_{\gamma}}(L_0) = \emptyset$ for all $\gamma$.

{\it Proof of Claim 2:} Our first observation is that $\mathcal{A}_{\hat{\theta}_{\gamma}}$ has no isolated points as we can perturb any simple space-like polygonal loop while keeping the angles the same. Because the $\gamma_P$ varies continuously with $P$, $\mathcal{A}_{\hat{\theta}_{\gamma}}(L)$ is an open subset and also has no isolated points. Let $0 < \theta_{\min}, \theta_{\max} < \pi$ be the maximum and minimum angles attained by $\theta_\cD: \cE(\eta) \to (0, \pi)$, respectively. Define the {\em coplanar locus} $\mathcal{AP}(L,\delta) \subset \mathcal{F}^n$ by $P =  \{P_i\}  \in \mathcal{AP}(L,\delta)$ if and only if 
\begin{itemize}
\item $P$ is a simple space-like polygonal loop in $\mathcal{F}$.
\item $e_i(P)$ and $e_j(P)$ are coplanar for all $i,j$.
\item $\gamma_P$ has length $\leq L$.
\item $\theta_{\min} \leq \theta_i(P) \leq \theta_{\max}$ for all $i$.
\item $ 2\pi + \delta \leq \sum_{i = 1}^{n-1} \theta_i(P)$.
\end{itemize}

Notice that we do not fix the angles in the definition of $\mathcal{AP}(L,\delta)$. We will now show that for every $\eps > 0$ there is an $L_\eps > 0$ such that $\mathcal{A}_{\hat{\theta}_{\gamma}}(L_\eps)$ lies in an $\eps$ neighborhood of $\mathcal{AP}(L_\eps, 0)$. Notice that $\mathcal{AP}(L, 0) \neq \emptyset$ for all $L$, so taking a neighborhood makes sense.

Fix $P = \{P_i\}  \in \mathcal{A}_{\hat{\theta}_{\gamma}}(L)$ and pick edges $e_i = e_i(P)$, $e_j = e_j(P)$ with $i \neq j$.  Since the pairs $e_i$ and $e_{i+1}$ are already coplanar for all $i$, we work with $|i - j| > 1$. We can assume that $e_i$ and $e_j$ share no endpoints, as otherwise they are coplanar. Let $X_{i,j}$ be the ideal tetrahedron obtained as the convex hull of $e_i \cup e_j$ in $\bH^3$. The fact that $P$ is simple and convexity imply that $X_{i,j} \subset \bH^3 \smallsetminus \bigcup_i P_i$. In particular, $\gamma_P$ is longer than the isotopic ``core'' curve $\gamma_{i,j}$ around $X_{i,j}$. Notice that $\gamma_{i,j}$ intersects two pairs of opposite edges of $X_{i,j}$. One pair being $\{e_i, e_j\}$ and the other we will call $\{e_i', e_j'\}$. Let $l_{i,j}$ and $l_{i,j}'$ be the complex lengths of the orthogonals for each pair, respectively. A simple computation using cross ratios shows that $1 = \cosh(l_{i,j}/2) \cosh(l_{i,j}'/2)$. However, since $0 < \Re(l_{i,j}) < L/2$ and $0 < \Re(l_{i,j}') < L/2$, we see that $$1 = \lim_{L \to 0} \cosh(l_{i,j}/2) \cosh(l_{i,j}'/2) = \cos(\Im(l_{i,j}/2)) \cos(\Im(l_{i,j}')/2).$$
Which implies that $\Im(l_{i,j}) \to 0$ as $L \to 0$. In particular, $e_i$ and $e_j$ become coplanar as $L \to 0$. Since $\mathcal{AP}(L,0)$ is the level set of $\Im(l_{i,j}) = 0$ over all $i,j$, we can always find $L_\eps > 0$ so that $\mathcal{A}_{\hat{\theta}_{\gamma}}(L_\eps)$ lies in an $\eps$ neighborhood of $\mathcal{AP}(L_\eps, 0)$. 

Returning to the statement of our Claim, assume that $\mathcal{A}_{\hat{\theta}_{\gamma}}(L)$ is non-empty for all $L > 0$. Pick a sequence $L_k \to 0$, then $\bigcap_k \overline{\mathcal{A}_{\hat{\theta}_{\gamma}}(L_k)}$ is non-empty and lies in $\mathcal{AP}(0, L)$ for all $L > 0$. In fact, we can do better. Let $\delta_{\hat{\theta}_{\gamma}} =  \theta_1 + \cdots +\theta_n - 2\pi$ and take $\delta_{\min} = \min_{\hat{\theta}_{\gamma}} \delta_{\hat{\theta}_{\gamma}}$ for all possible simple loops $\gamma$ on $\Sigma_{\cD}$ contractible in $S$. By Lemma \ref{lem:angles} and the fact that $\eta^*$ has finitely many edges, $\delta_{\min}$ is attained and $\delta_{\min} > 0$. It follows that $\bigcap_k \overline{\mathcal{A}_{\hat{\theta}_{\gamma}}(L_k)} \subset \mathcal{AP}(L, \delta_{\min})$ for all $L > 0$ and all $\gamma$ of interest. It remains to show that there exists $L_{\min}$ such that $\mathcal{AP}(L_{\min}, \delta_{\min}) = \emptyset$.

Let $\mathcal{AP}^o(L, \delta_{\min})$ be the set of $P \in  \mathcal{AP}(L, \delta_{\min})$ such that $e_i(P)$ and $e_j(P)$ do not share any endpoints for all $i \neq j$. Notice that  $\overline{\mathcal{AP}^o(L,  \delta_{\min})} = \mathcal{AP}(L, \delta_{\min})$, so we may work with $\mathcal{AP}^o(L,  \delta_{\min})$.  

Consider the Klein model for $\bH^3$ embedded in $\bR \bP^3$ and pick $P \in \mathcal{AP}^o(L,  \delta_{\min})$. The lines $e_i(P)$ in the Klein model extend to lines in $\bR \bP^3$ that are still pairwise coplanar. Projective geometry tells us that if $\{e_i(P)\}$ do not all lie in a $2$-plane in $\bR \bP^3$, then they must all intersect at one point. Since $\theta_i(P) \geq \theta_{\min} > 0$, our edges cannot lie in a $2$-plane. Further, since they are disjoint in $\bH^3$, their common point of intersection lies outside $\bH^3$ in $\bR \bP^3$. This common point of intersection is dual to a hyperbolic plane $Q$ that is perpendicular to all $e_i(P)$. Since $T_P$ is embedded, $Q \cap T_P$ is a convex hyperbolic polygon with exterior angles $\theta_i(P)$ and with perimeter of length $\leq L$. Gauss-Bonnet gives that $\mathrm{area}(Q \cap T_P) + 2\pi = \sum_i \theta_i(P)$ and therefore $\mathrm{area}(Q \cap T_P) \geq \delta_{\min}$. However, by the isoperimetric inequality, we can find $L_{\min}$ such that $\mathrm{area}(Q \cap T_P) < \delta_{\min}$. Therefore, there exists exists $L_{\min}$ such that $\mathcal{AP}(L_{\min}, \delta_{\min}) = \emptyset$.

It follows that for some $L_0 > 0$, $\mathcal{A}_{\hat{\theta}_{\gamma}}(L_0) = \emptyset$ for all $\gamma$. \qed

Taking Claim 1 and Claim 2 together, we obtain the desired result.
\end{proof}

For non-contractible  curves, we have the following weaker result.

\begin{prop} \label{prop:non-contractible}
Let $(\eta, \theta)$ be an admissible pair for $S$ and let $K\subset \cT$ be a compact subset. There exists a constant $L_1>0$ depending only on $(\eta, \theta)$ and $K$ such that if $(\sigma,\cD)\in \cC_{\eta,\theta}$ with $f(\sigma)\in K$ and $\gamma$ is closed geodesic on $\Sigma_{\cD}$ which is not contractible in $S$, then the length of $\gamma$ is at least $L_1$.
\end{prop}

\begin{proof}
Let $\eta, \theta$ be fixed. We begin by showing that the set $M = \{ h_\sigma \mid \sigma \in \cC_{\eta,\theta}\} \subset \cT$ is relatively compact. Recall that $h_\sigma$ is the metric on the finite boundary of the hyperbolic end defined by $\sigma$ and $\mu_\sigma$ is the measured bending lamination.

Recall that the {\em grafting map} $\Gr : \cT \times \cML \to \cC$ is a homeomorphism that gives rise to the continuous surjection $\gr : \cT \times \cML \to \cT$ given by $\gr = f \circ \Gr$, see \cite{dumas-survey, thurston-notes, kamishima-tan} for details. Consider the map $G : \cT \times \cML \to  \cT \times \cML$ given by $G(w,\mu) = \left(\gr(w, \mu), \mu \right)$. Scannell and Wolf \cite{scannell-wolf} showed that for fixed $\mu \in \cML$, the map $\gr_\mu = \gr(\cdot, \mu)$ is a diffeomorphism and therefore $G$ is a continuous bijection. Since $\cML$ and $\cT$ are both homeomorphic to $\bR^{6g-g}$, invariance of domain tells us that $G$ is a homeomorphism.

By Lemma \ref{lem:lambda}, there exists a compact subset $K_\cML \subset \cML$ such that for all $(\sigma,\cD) \in \cC_{\eta,\theta}$, one has $\mu_\sigma \in K_\cML$. By construction, $\gr(h_\sigma, \mu_\sigma) = f \circ \Gr(h_\sigma, \mu_\sigma) \in K$ for all $\sigma \in \cC_{\eta,\theta}$, so $M \subset \Pi_1(G^{-1}(K \times K_\cML))$, where $\Pi_1$ is the projection to the $\cT$ factor. Since $\Pi_1 \circ G^{-1}$ is continuous, $M$ is relatively compact.

This implies that there exists a constant $L_1>0$ such that for all $h_\sigma \in M$, the lengths of all closed geodesics in $(S, h_\sigma)$ are bounded from below by $L_1$.

 For $(\sigma,\cD)\in \cC_{\eta,\theta}$, let $\gamma$ be a closed geodesic on $\Sigma_\cD$ which is not contractible in $S$. The orthogonal projection from $\Sigma_\cD$ to $\partial_0 E(\sigma)$ is a contraction and, therefore, the image of $\gamma$ by this orthogonal projection has length at least $L_1$. It follows that the length of $\gamma$ on $\Sigma_\cD$ is also at least equal to $L_1$.
\end{proof}

\subsection{Proof of the main result}

Fix an admissible pair $(\eta, \theta)$ on $S$. We consider a compact subset  $K$ of $\cT$, and let $K_\cC$ be the projection to $\cC$ of $f_{\eta,\theta}^{-1}(K)$.

\begin{lemma} \label{lem:K_C}
 $K_\cC$ is relatively compact in $\cC$. 
\end{lemma}

\begin{proof} Using the notation from the proof of Proposition \ref{prop:non-contractible}, it is clear that $K_C \subset \Gr(G^{-1}(K \times K_\cML))$. Since $\Gr \circ G^{-1}$ is continuous, $K_\cC$ is relatively compact.
\end{proof}

\begin{definition}
Given $\sigma\in \cC$ and a round disk $D$ in $(\wt{S},\wt{\sigma})$, the {\it radius} $r_\sig(D)$ of $D$ is defined as the radius of the smallest hyperbolic disk in $(\wt{S}, h)$ containing $D$, where $h$ is the hyperbolic metric compatible with the complex structure of $\wt{\sigma}$. Notice that the radius lies in $(0, \infty]$, as any round disk touching the limit set of $\sig$ has infinite radius.
\end{definition}

\begin{lemma}
  There exists $r_0>0$ such that for any $(\sigma, \cD)\in f_{\eta,\theta}^{-1}(K)$ and for any disk $D\in \cD$, the radius of $D$ is at least $r_0$.
\end{lemma}

\begin{proof}
Suppose that there is no such lower bound and take a sequences $(\sigma_n,\cD_n)\in f_{\eta,\theta}^{-1}(K)$ and $D_n \in \cD_n$ such that $r_{\sig_n}(D_n) \to 0$ as $n \to \infty$. Let $h_n$ be the hyperbolic metric in the conformal class defined by the complex structure of $\sigma_n$. Up to extracting a subsequence, we can suppose that $\sigma_n \to \sigma\in \cC$ by Lemma \ref{lem:K_C} and $D_n$ is corresponds to a fixed face of $\eta$. Since $r_{\sig_n}(D_n) \to 0$, the subsets $D_n \subset \sig_n$ converge to a point $x_\infty \in \sig$.

  Let $E_n = E(\sigma_n)$ be the hyperbolic end associated to $\sigma_n$ and note that $E_n$ converges to $E = E(\sigma)$ in the Gromov-Hausdorff topology. Let $P_n\subset E_n$ be the ideal polyhedron corresponding to the Delaunay circle pattern $\cD_n$. The face $f_n$ of $P_n$ corresponding to $D_n$ converges to the point $x_\infty \in \partial_\infty E$. Therefore, there is a set $Z$ of at least 3 vertices of $P_n$ converging to the same point $x_\infty$.

Let $\gamma$ be a simple closed curve, contractible in $S$, separating $Z$ from the other vertices of $P_n$. Since the $f_n$ converge to an ideal point, the length of $\gamma$ for $h_n$ converges to $0$ as $n \to \infty$. However, this contradicts Proposition \ref{prop:contractible}. Therefore, some lower bound on the radii must exists depending only on $(\eta, \theta)$.
\end{proof}

\begin{definition}\label{def:bundle}
For all $\sigma\in \cC$, let $\Delta_\sigma$ be the space of disks in $(S,\sigma)$ where the topology is given by the Hausdorff distance on closed sets in the hyperbolic metric $h$ compatible with  the complex structure of  $\wt{\sig}$.
\end{definition}

The space $\Delta_\sigma$ is a manifold as seen from the next Lemma.

\begin{lemma}\label{lem:radius} Let $T_1S = (TS \smallsetminus S \times \{0\})/\bR_{>0}$ be the bundle of tangent directions on $S$. For each $\sigma \in \cC$, there is a continuous surjection $p_\sig : T_1 S \times [0, \infty] \to \Delta_\sigma$ such that the radius of $p(x,a)$ is $a$. 
\end{lemma}

\begin{proof} The map $p_\sig$ is defined as follows. An element $x \in T_1(S)$ is given by a point $q \in S$ and a ray $\ell$ of tangent vectors through $q$. Consider $$A_x = \{D \in \Delta_\sig \mid q \in \partial D \text{ and the outward normal to } D \text{ at } q \text{ lies in } \ell\}.$$ Since $\dev_\sig(D)$ is a disk in $\bCP^1$, the set $A_x$ is determined by one real parameter. In fact, $A_x$ is a one-parameter family of nested closed disks. The maximal disk in this family has radius $\infty$ and for distinct $D^0, D^1 \in A_x$ with $D^0 \subset D^1$, we have $r_\sig(D^0) < r_\sig(D^1)$. We can the therefore parametrize $A_x$ by $\left.r_\sig\right|_{A_x}^{-1}$. Let $p_\sig(x,a)$ be the disk $\left.r_\sig\right|_{A_x}^{-1}(a)$. By the definition of $r_\sig$, $p_\sig$ is a continuous surjection.
\end{proof}

The fibers of $p_\sig$ are circles which give a simple foliation of $T_1S \times (0,\infty]$. Thus, $\Delta_\sigma$ can be see as the leaf-space of this foliation. 

Since continuous deformations $\sig_t \in \cC$ give continuous deformations of the hyperbolic metrics $h_t$ on $\wt{S}$ compatible with $\wt{\sig}_t$, there is a natural homeomorphism $\Delta_{\sigma_0} \cong \Delta_{\sig_t}$ and we can define the trivial bundle $\Delta=\cup_{\sigma\in \cC}\Delta_\sigma$.

\begin{cor}
 There is a compact subset $K_\Delta$ of $\Delta$ containing $(\sigma,D)$ for all disks $D\in \cD$ when $(\sigma,\cD)\in f_{\eta,\theta}^{-1}(K)$.
\end{cor}

\begin{proof}
 
The key observation is that for all $\sigma \in \cC$, the space of embedded round disks in $(S,\sigma)$ of radius at least $r_0 > 0$ is compact. Indeed, by Lemma \ref{lem:radius}, this set is $p_\sig\left(T_1S \times [r_0,\infty]\right)$, which is compact by the continuity of $p_\sig$.

As a consequence, since $K_\cC$ is compact, the space of embedded open disks in $(S,\sigma)$ for all $\sigma\in K_\cC$ is again compact. However we already know that if $(\sigma,\cD)\in f_{\eta,\theta}^{-1}(K)$ and $D\in \cD$ then the radius of $D$ is at least $r$. The result follows.
\end{proof}

\begin{cor}
  $f_{\eta,\theta}^{-1}(K)$ is compact.
\end{cor}

\begin{proof}
  The space of possible circle patterns with the given combinatorics, intersection angles, and underlying complex projective structure in $K_C$, is compact, and therefore $f_{\eta,\theta}^{-1}(K)$ is compact.
\end{proof}

\section{Acknowledgements.}
The first author was partially supported by University of Luxembourg IRP NeoGeo and by FNR projects INTER/ANR/15/11211745 and OPEN/16/11405402. He also acknowledges support from U.S. National Science Foundation grants DMS-1107452, 1107263, 1107367 ``RNMS: GEometric structures And Representation varieties'' (the GEAR Network). Both authors would like to thank Ser Peow Tan for useful discussions and motivation.

\bibliographystyle{plain}
\bibliography{biblio}

\end{document}